\documentclass[11pt]{amsart}
\usepackage{latexsym}
\usepackage{amssymb}
\usepackage{amsmath}

\newtheorem{theorem}{Theorem}[section]
\newtheorem{lemma}[theorem]{Lemma}

\newtheorem{corollary}[theorem]{Corollary}
\newtheorem{definition}[theorem]{Definition}
\newtheorem{defn}[theorem]{Definition}
\newtheorem{prop}[theorem]{Proposition}
\newtheorem{thm}[theorem]{Theorem}
\newtheorem{cor}[theorem]{Corollary}

\theoremstyle{definition}
\newtheorem{remark}[theorem]{Remark}

\newcommand\qc{\mathop{\rm qc}}
\newcommand\q{\mathop{\rm q}}
\newcommand\qs{\mathop{\rm qs}}
\newcommand\vect{\mathop{\rm vect}}
\newcommand\f{\mathop{\rm f}}
\newcommand\qa{\mathop{\rm qa}}
\newcommand\tr{\mathop{\rm tr}}
\newcommand\cc{\mathop{\rm qc}}
\newcommand\loc{\mathop{\rm loc}}
\newcommand\xx{\mathop{\rm x}}

\newcommand{\cl}[1]{\mathcal{#1}}
\newcommand{\bb}[1]{\mathbb{#1}}

\newcommand{\abs}[1]{\ensuremath{{\lvert#1\rvert}}}

\newcommand{\Tr}[0]{\ensuremath{\textnormal{Tr}}}

\newcommand{\ot}[0]{\otimes}
\newcommand{\Gc}[0]{{\overline{G}}}

\newcommand{\thpbar}[0]{{\overline{\vartheta}^+}}

\newcommand{\xisdp}[0]{{\xi_{\textnormal{SDP}}}}

\newcommand{\homL}[0]{\stackrel{\loc}{\to}}
\newcommand{\homX}[0]{\stackrel{{\rm x}}{\to}}

\newcommand{\homq}[0]{\stackrel{q}{\to}}

\newcommand{\myvec}[1]{\mathbf{#1}}
\newcommand{\evec}[0]{\myvec{e}}

\newcommand{\djp}[0]{*}

\begin{document}

\title[Quantum chromatic numbers]{Estimating quantum chromatic numbers}

\author[V.~I. Paulsen]{Vern I. Paulsen}
\address{Department of Mathematics, University of Houston,
Houston, Texas 77204-3476, U.S.A.}
\email{vern@math.uh.edu}

\author[S.~Severini]{Simone Severini}
\address{Department of Computer Science, University College London,
Gower Street, London WC1E 6BT, United Kingdom}
\email{s.severini@ucl.ac.uk}

\author[D.~Stahlke]{Daniel Stahlke}
\address{Department of Physics, Carnegie Mellon University, Pittsburgh, Pennsylvania 15213, USA}
\email{dan@stahlke.org}

\author[I.~G. Todorov]{Ivan G. Todorov}
\address{Pure Mathematics Research Centre, Queen's University Belfast, Belfast BT7 1NN, United Kingdom}
\email{i.todorov@qub.ac.uk}

\author[A.~Winter]{Andreas Winter}
\address{ICREA and F\'{\i}sica Te\`{o}rica: Informaci\'{o} i Fenomens Qu\`{a}ntics, Universitat Aut\`{o}noma de Barcelona, ES-08193 Bellaterra (Barcelona), Spain}
\email{andreas.winter@uab.cat}

\thanks{This work supported in part by NSF (USA), EPSRC (United Kingdom), the Royal Society, the European Commission (STREP ``RAQUEL''), the ERC (Advanced Grant ``IRQUAT''), the Spanish MINECO (project FIS2008-01236, with the support of FEDER funds), and the Isaac Newton Institute for Mathematical Sciences during the semester \emph{Mathematical Challenges in Quantum Information}, Aug-Dec 2014.}

\keywords{operator system, tensor product, chromatic number}
\subjclass[2010]{Primary 46L07, 05C15; Secondary 47L25, 81R15}

\date{10 September 2015}

\begin{abstract}
We develop further the new versions of quantum chromatic numbers of
graphs introduced by the first and fourth authors.
We prove that the problem of computation of the commuting quantum chromatic number
of a graph is solvable by an SDP  algorithm and describe an hierarchy of variants of the
commuting quantum chromatic number which converge to it.
We introduce the tracial rank of a graph, a parameter that gives a lower bound for the commuting
quantum chromatic number and parallels the projective rank, and prove that it is multiplicative.
We describe the tracial rank, the projective rank and the fractional chromatic numbers in a
unified manner that clarifies their connection with the commuting quantum chromatic number,
the quantum chromatic number and the classical chromatic number, respectively.
Finally, we present a new SDP algorithm that yields a parameter
larger than the Lov\'asz number and is yet a lower bound for the tracial rank of the graph.
We determine the precise value of the tracial rank of an odd cycle.
\end{abstract}

\maketitle

\section{Introduction}

We assume that the reader is familiar with some concepts from graph theory and refer the reader to the text \cite{gr} for any terminology that we do not explain.

In \cite{galliardwolf,AHKS06,cnmsw,SS12} the concept of the quantum chromatic number $\chi_{\q}(G)$
of a graph $G$ was developed and inequalities for estimating this parameter,
as well as methods for its computation, were presented.
In \cite{pt_chrom} several new variants of the quantum chromatic number, 
denoted $\chi_{\qc}(G)$, $\chi_{\qa}(G)$ and $\chi_{\qs}(G)$, were introduced,
as well as $\chi_{\vect}(G)$.  
The motivation behind them came from 
conjectures of Tsirelson and Connes and the fact that the set of correlations of 
quantum experiments may possibly depend on which set of quantum mechanical axioms one 
chooses to employ. Given a graph $G$, the aforementioned chromatic numbers satisfy the inequalities
\[
  \chi_{\vect}(G) \le \chi_{\qc}(G) \le \chi_{\qa}(G) \le \chi_{\qs}(G) \le \chi_{\q}(G) \le \chi(G),
\]
where $\chi(G)$ denotes the classical chromatic number of the graph $G$.

The motivation of \cite{pt_chrom} for defining and studying 
these new chromatic numbers comes from the fact that 
if Tsirelson's conjecture is true, then $\chi_{\qc}(G) = \chi_{\q}(G)$ for every graph $G$, while if 
Connes' Embedding Problem has an affirmative answer, 
then $\chi_{\qc}(G) = \chi_{\qa}(G)$ for every graph $G$.  
Thus, computing these invariants gives a means to test the corresponding conjectures.

In \cite{cmrssw} it was shown that
\[ \lceil \vartheta^+(\overline{G}) \rceil = \chi_{\vect}(G),\]
where $\lceil r \rceil$ denotes the least integer greater than or equal to
$r$ and $\vartheta^+$ is
Szegedy's~\cite{365707} variant of Lov\'asz's~\cite{lo} $\vartheta$-function. 
Furthermore, this identity was used to give the first example of a graph for which 
$\chi_{\vect}(G) \ne \chi_{\q}(G)$.  Also, since $\vartheta^+$ is defined by an SDP, 
the aforementioned result shows that $\chi_{\vect}(G)$ is computable by an SDP.

In this paper we show that for each size of graph,
$\chi_{\qc}(G)$ is also computable by an SDP.
Our proof builds on ideas borrowed from the \lq\lq NPA hierarchy'' 
exhibited in \cite{npa}. It uses a compactness argument to show that for the 
purposes of computing this integer the hierarchy terminates, but does not yield 
the stage at which it does so. 
Thus, while we can say that it is computable by one of the SDP's in the hierarchy, 
we cannot explicitly determine the size of this SDP.  
It is known that $\chi(G)$ is computable by an SDP, but it is still not known if 
the same is true for $\chi_{\qa}(G), \chi_{\qs}(G)$ and $\chi_{\q}(G)$. 
If the Tsirelson and Connes conjectures hold true 
then these three quantities must also be computable by SDP's as they
then are all equal to $\chi_{\qc}(G)$.

D.~Roberson and L.~Man\v{c}inska~\cite{arxiv:1212.1724} introduced a Hilbert 
space variant of the fractional chromatic number $\chi_{\f}(G)$ of a graph $G$, called the 
projective rank, and denoted $\xi_{\f}(G)$. 
They proved that $\xi_{\f}(G)$ is a lower bound for $\chi_{\q}(G)$;
this estimate was critical for identifying a graph $G$ with 
$\chi_{\vect}(G) \ne \chi_{\q}(G)$.  
However, it is still not known if $\xi_{\f}(G)$ is a lower bound for 
the variants of the quantum chromatic number studied in \cite{pt_chrom}.

In the present paper, we introduce a new variant of the projective rank, which we call 
\emph{tracial rank} and show that it is a lower bound for $\chi_{\qc}(G)$. 
We also give a new interpretation of the projective rank in terms of traces on finite dimensional C*-algebras. 
These parameters and their properties allow us to give the first example of a graph for which 
$\chi_{\vect}(G) \ne \chi_{\qc}(G)$.

Finally, we present a new SDP that yields a parameter larger than Szegedy's bound 
(i.e., it is an \emph{SDP relaxation} of the latter), and which is still a lower 
bound on $\chi_{\qc}(G)$. En route, we show that our tracial rank is multiplicative. 

To put our work into a broader context, recently a number of graph parameters, including clique, chromatic, and independence numbers,
have been generalized to \lq\lq quantum'' versions by relating the respective
number to attaining maximum probability $1$ in a so-called non-local game.
The present work shows that some of the techniques to bound and separate these
numbers, originally developed for the exact, finite-dimensional and tensor
product case, can be extended to the approximate and relativistic setting,
and that they can be bounded by SDP's without going directly to the NPA hierarchy.

\section{Quantum correlations and variants of the quantum chromatic number}
\label{s_qdm}
In this section we summarize the properties of some of the variants of the quantum chromatic number 
introduced in \cite{pt_chrom} from the viewpoint of correlations and derive a few 
additional properties of the corresponding models that will be essential later.
We introduce a new notation that we hope serves to clarify and unify many of the ideas from \cite{pt_chrom}.

Let $G = (V,E)$ be a graph; here, $V = V(G)$ is the set of vertices, while 
$E = E(G)$ is the set of edges, of $G$. If $(v,w)\in E$, we write $v\sim w$. 
In \cite{AHKS06} and \cite{cnmsw}, the authors considered a \lq\lq graph colouring game'', 
where two players, Alice and Bob, try to convince a referee that they have a colouring of $G$;
the referee inputs a pair $(v,w)$ of vertices of $G$, and each of the players
produces an output, according to a
previously agreed \lq\lq quantum strategy'', that is, a probability distribution
derived from an entangled state and collections of positive operator-valued measures, POVM's.
To formalise this, recall that a POVM is a collection $(E_i)_{i=1}^k$ of positive operators
acting on a Hilbert space $H$ with $\sum_{i=1}^k E_i = I$ (as usual, here we denote by $I$
the identity operator). If, in addition, each $E_i$ is a projection, then the collection is called a projection-valued measure, or PVM.  When $H = \bb{C}^p$ is finite dimensional,
we identify the operators on $H$ with elements of the algebra $M_p$ of all
$p \times p$ complex matrices.
Given POVM's $(E_{v,i})_{i=1}^c \subseteq M_p$ and
$(F_{w,j})_{j=1}^c \subseteq M_q$, where $v,w \in V$, and
a unit vector $\xi\in \bb{C}^p\otimes\bb{C}^q$, one associates with each pair
$(v,w)$ of vertices of $G$ the probability distribution
\begin{equation}
  \label{eq_qq}
  p(i,j|v,w)=\langle (E_{v,i} \otimes F_{w,j}) \xi, \xi \rangle, \, 1 \le i,j \le c,
\end{equation}
where,
for an input $(v,w)$ from the referee, $\langle (E_{v,i} \otimes F_{w,j}) \xi, \xi \rangle$
is the probability for Alice producing an output $i$ and Bob -- an output $j$.

The set of all $nc \times nc$ matrices that are obtained by allowing $p$ and $q$ to vary through all the natural numbers, is called the set of
\emph{quantum correlations} and is generally denoted $Q(n,c)$; when $n$ and $c$ 
are clear from the context, we simply write $Q$. For reasons that will be clear shortly, we shall adopt the notation:  $C_{\rm q}(n,c) = Q(n,c).$

Alice and Bob can thus convince the referee that they have a $c$-colouring if the
following conditions are satisfied:
\begin{multline}\label{qcdefn}
\forall v, \forall i \ne j, \langle (E_{v,i} \otimes F_{v,j}) \xi, \xi \rangle =0,\\
\forall (v,w) \in E, \forall i, \langle (E_{v,i} \otimes F_{w,i}) \xi, \xi \rangle =0.
\end{multline}
If this happens, we say that the graph $G$ admits a \emph{quantum $c$-colouring};
the smallest positive integer $c$ for which $G$ admits a quantum $c$-colouring was
called in \cite{cnmsw} the \emph{quantum chromatic number} of $G$ and denoted by $\chi_{\q}(G)$.

One can interpret the quantum chromatic number in terms of a linear functional on the set of quantum correlations.

\begin{defn} Let $G=(V,E)$ be a graph,
let $|V|=n$ and fix $c \in \bb N$. The graph functional, $L_{G,c}: M_{nc} \to \bb C$ is defined by
\[ L_{G,c}\big( (a_{v,i,w,j}) \big) = \sum_{i \ne j, v} a_{v,i,v,j} + \sum_{i, v\sim w} a_{v,i,w,i}. \]
\end{defn}

\begin{prop}\label{p_LGc}
Let $G$ be a graph on $n$ vertices. Then 
$\chi_{\q}(G) \le c$ if and only if there exists $A \in C_{\rm q}(n,c)$ such that $L_{G,c}(A) = 0$.
\end{prop}
\begin{proof} 
Since the entries of every correlation in $C_q(n,c)$ are non-negative, 
it follows that $L_{G,c}(A) =0$ if and only if $a_{v,i,w,j} =0$ whenever $v=w$ and $i \ne j$ and 
whenever $(v,w)$ is an edge and $i=j$. Thus, there exists $A \in C_{\rm q}(n,c)$ 
precisely when conditions (\ref{qcdefn}) are satisfied.
\end{proof}

Several variants of the quantum chromatic number were introduced in \cite{pt_chrom}. We focus on two 
denoted by $\chi_{\qc}(G)$ and $\chi_{\qa}(G)$, and 
called the \emph{commuting quantum chromatic number}, and the 
\emph{approximate quantum chromatic number}, respectively. 
Both of them can be obtained as in Proposition \ref{p_LGc} by 
varying the sets of correlations that can be considered in place of $C_{\rm q}(n,c)$.

Let $C_{\rm qa}(n,c) :=C_{\rm q}(n,c)^-$ be the closure of $C_{\rm q}(n,c)$; we note that 
it is not known if $C_{\rm q}(n,c)$ is closed. 
In \cite{pt_chrom} it was shown that 
$C_{\rm qa}(n,c)$ can be identified with the state space of a certain minimal 
tensor product and that consequently, $\chi_{\qa}(G) = \chi_{\rm qmin}(G)$, 
where this latter quantity was originally given a different definition.

We let $C_{\cc}(n,c)$ denote the set of correlations 
obtained using relativistic quantum field theory. 
To be precise, instead of assuming that the POVM's $(E_{v,i})$ and $(F_{w,j})$ 
act on two finite dimensional Hilbert spaces and forming the tensor product of those spaces, 
we assume instead that they act on a common, possibly infinite dimensional, Hilbert space 
and that the $E$'s and $F$'s mutually commute. 
Thus, $C_{\cc}(n,c)$ is the set of all $nc \times nc$ matrices of the form
\[ p(i,j|v,w) = \big( \langle E_{v,i}F_{w,j} \xi, \xi \rangle \big)_{v,i,w,j},\]
where $E_{v,i}, F_{w,j} $ are positive operators on a Hilbert space $\cl H$ satisfying

\begin{multline}\label{commprop}
\sum_{i=1}^c E_{v,i} = \sum_{j=1}^c F_{w,j} = I, \, \forall \, v,w,\\
E_{v,i}F_{w,j} = F_{w,j}E_{v,i}, \, \forall \, v,w,i,j,
\end{multline}
and $\xi \in \cl H$ is a unit vector.

As in the finite dimensional case, by enlarging the Hilbert space,  
one may assume without loss of generality, 
that all of these operators are orthogonal projections (see \cite[Theorem 2.9]{pt_chrom}
and Theorem \ref{findimrealisation} below).

Finally,  ${\rm Loc}(n,c)$ denotes the 
set of all \emph{classical correlations}, also called \emph{local correlations}, that is, the 
set of matrices $(p(i,j|v,w))_{v,i,w,j}$ 
which are in the closed convex hull of the product distributions $p(i,j|v,w) = p^1(i|v)p^2(j|w)$ where  $p^1(i|v) \ge 0$ and $\sum_{i=1}^c p(i|v) =1$ is a set of $c$-outcome probability distributions indexed by $v \in V$ with $|V|=n$ and, similarly, $p^2(j|w)$ is a set of $c$-outcome probability distributions indexed by $w \in V$. Since both of these sets of probability distributions define compact sets in $\bb R^{nc}$ by Caratheodory's theorem, given any element of ${\rm Loc}(n,c)$, there exist at most $n^2c^2+1$ probability distributions $(p_k^1(i|v))_{i=1}^c$, $v\in V$ 
(resp. $(p_k^2(j|w))_{i=1}^c$, $w\in V$), and non-negative scalars $\lambda_k$ 
such that 
\begin{equation}\label{eq_localco}
  p(i,j|v,w) = \sum_{k} \lambda_k p_k^1(i|v)p_k^2(j|w), \ \ i,j=1,\dots,c, \ v,w\in V.
\end{equation}
For consistency of notation, we set $C_{\rm loc}(n,c) := {\rm Loc}(n,c).$

There is another useful characterisation of the set $C_{\rm loc}(n,c)$, discussed in \cite[p. 5]{pt_chrom}. Let $\cl D$ denote the tensor product of $n$ copies of the abelian C*-algebra $\ell^{\infty}_c$, 
which is *-isomorphic to the space of all (continuous) functions on $nc$ points and set $e_{v,i}= 1 \otimes \cdots 1 \otimes e_i \otimes 1 \cdots 1,$ where $e_i$ denotes the $i$-th canonical basis vector for $\ell^{\infty}_c$ and it occurs in the $v$-th term of the tensor product, $1 \le v \le n$. Then $(p(i,j|v,w)) \in C_{loc}(n,c)$ if and only if there is a state $s: \cl D \otimes \cl D \to \bb C$ such that $p(i,j|v,w) = s(e_{v,i} \otimes e_{w,j}).$  To see this fact, note that in the above formula each $p_k^1(i|v)=s_k^1(e_{v,i})$ for a state $s_k^1: \cl D \to \bb C$, while $p_k^2(j|w)= s_k^2(e_{w,j}).$ Thus, (\ref{eq_localco}), for a typical element of $C_{\rm loc}(n,c)$, becomes $p(i,j|v,w) = \sum_k \lambda_k s_k^1 \otimes s_k^2(e_{v,i} \otimes e_{w,j}),$ which is a convex combination of product states.  The fact that convex combinations of product states yields all states on $\cl D \otimes \cl D$ follows from another application of Caratheodory's theorem.

It is easy to see that
\[
  C_{\rm loc}(n,c) \subseteq C_{\rm q}(n,c) \subseteq C_{\rm qa}(n,c) \subseteq C_{\cc}(n,c).
\]

Note that the correlations belonging to 
$C_{\rm loc}(n,c)$ can be realised as in (\ref{eq_qq}), 
but with the POVM's $(E_{v,i})_{i=1}^c$ and $(F_{w,j})_{j=1}^c$ consisting of 
mutually commuting operators. 

These various sets of correlations allowed \cite{pt_chrom} to generalise and unify the definitions of the  quantum chromatic number of a graph $G$ (on $n$ vertices) by setting $\chi_{\rm x}(G)$ equal to the least integer $c$ such that there
 exists $p(i,j|v,w) \in C_{\rm x}(n,c)$ satisfying:
 
 \begin{multline}\label{tdefn}
\forall v, \forall i \ne j, \,\, p(i,j|v,v) =0,\\
\forall (v,w) \in E, \forall i, \,\, p(i,i|v,w) =0,
\end{multline}
where ${\rm x} \in \{{\rm loc},{\rm q},{\rm qa},{\rm qc}\}.$
In \cite{cnmsw} it is shown that $\chi_{\rm loc}(G)$ is equal to the usual chromatic number of a graph, $\chi(G).$

Since $p(i,j|v,w) \ge 0$ for elements of each of these correlation sets, the proof of the following is identical to the proof of the last proposition:

\begin{prop}\label{p_chqaqc}
Let $G$ be a graph on $n$ vertices and let ${\rm x} \in \{{\rm loc},{\rm q},{\rm qa},{\rm qc}\}.$ 
Then $\chi_{\rm x}(G) \le c$ if and only if there exists $A \in C_{\rm x}(n,c)$ such that $L_{G,c}(A)=0$.
\end{prop}

From the above containments and proposition, we immediately see \cite{pt_chrom} that
\[
  \chi_{\qc}(G) \le \chi_{\qa}(G) \le \chi_{\q}(G) \le \chi(G).
\]

\begin{remark} 
By Proposition \ref{p_chqaqc}, and the fact that  ${\rm Loc}(n,c)$ is compact and convex, each graph $G$ with 
$\chi_{\q}(G) =c < \chi(G)$ yields a graph functional $L_{G,c}$ 
that is strictly positive on the set $C_{\rm loc}(n,c)={\rm Loc}(n,c)$ of  classical correlations
and vanishes on the set $C_{\rm q}(n,c) = Q(n,c)$. Thus, each such graph gives a functional 
$L_{G,c}$ that gives a Bell-type inequality separating local from quantum.
\end{remark}

\begin{remark} By Proposition \ref{p_chqaqc}(1), determining if the minimum of $L_{G,c}$ over the polytope ${\rm Loc}(n,c)$ is 0, gives a LP to determine if $\chi(G)\le c$.
\end{remark}

\begin{remark} 
The statement $C_{\rm q}(n,c) = C_{\cc}(n,c)$ for all $n$ and $c$ is often referred to as 
the \emph{(strong bivariate) Tsirelson conjecture}. 
Thus, if the Tsirelson conjecture holds true, then necessarily, 
$\chi_{\rm c}(G) = \chi(G)$ for every graph $G$.  
In addition, work of N. Ozawa \cite{oz} shows that 
the Connes Embedding Conjecture is equivalent to the statement that 
$C_{\rm qa}(n,c) = C_{\cc}(n,c)$ for all $n$ and $c$. Thus, if the Connes Embedding Conjecture
holds true, then necessarily $\chi_{\qc}(G) = \chi_{\qa}(G)$ for every graph $G$.
We shall refer to the equality 
$C_{\rm qa}(n,c) = C_{\rm q}(n,c)$, $ \forall n,c\in \bb{N}$, as the \emph{closure conjecture}. 
Thus, if the closure conjecture is true, then $\chi_{\q}(G) = \chi_{\qa}(G)$ for every graph $G$.

Hence, determining if these chromatic numbers are always equal or 
can be separated would shed some light on these two conjectures. 
This was the original motivation  
for introducing these new parameters in \cite{pt_chrom}.
\end{remark}

\section{A hierarchy of chromatic numbers}
\label{s_qrc} 
In this section we revisit the ideas of the NPA hierarchy \cite{npa} 
and use them to construct a descending sequence of state spaces $C^N(n,c)$ that converges 
in an appropriate sense to $C_{\cc}(n,c)$.  
These will allow us to construct a sequence of ``chromatic'' numbers 
$\chi^N_{\qc}(G)$ each of which is computable by an SDP and which converges to 
$\chi_{\qc}(G)$.  To accomplish this, we need to review a certain C*-algebra 
intrinsic in the definition of $\chi_{\qc}(G)$.

Let $n,c\in \bb{N}$,
$F(n,c) = \bb Z_c * \cdots * \bb Z_c$ ($n$ copies), where $\bb Z_c = \{0,1,\dots,c-1\}$
is the cyclic group with $c$ elements, and
$C^*(F(n,c))$ is the (full) C*-algebra of $F(n,c)$.
The C*-algebra of $\bb{Z}_c$ is canonically *-isomorphic, via Fourier transform, to the (abelian) C*-algebra
$\ell^{\infty}_c = \{(\lambda_i)_{i=1}^c : \lambda_i\in \bb{C}, i = 1,\dots,c\}$.
Thus, $C^*(F(n,c))$ is canonically *-isomorphic to the free product  C*-algebra, amalgamated over the unit,
$\ell^{\infty}_c\ast_1\cdots\ast_1 \ell^{\infty}_c$ ($n$ copies).

We denote by $V$ the set $\{1,2,\dots,n\}$.
Let $e_{v,i}$ denote the element of
$C^*(F(n,c))$ that is in the $v$-th copy of $\ell^{\infty}_c$ and is the vector that is $1$ in the $i$-th component and $0$ elsewhere.
Alternatively, if we regard $F(n,c)$ as generated by unitaries, $u_v$ with $u_v^c=1$, then $e_{v,j}$ corresponds to the spectral projection of $u_v$ onto the eigenspace of $\omega^j$ where $\omega = e^{2 \pi i/c}.$ In particular, $e_{v,j}$ belongs to the group algebra of $F(n,c)$ and because $u_v = \sum_{i=1}^c \omega^i e_{v,i}$ the collection $\{ e_{v,i}: 1 \le i \le c, v \in V \}$ is another set of generators of the group algebra.

We let
$\cl S(n,c) = \ell^{\infty}_c \oplus_1 \cdots \oplus_1 \ell^{\infty}_c$
($n$ copies) be the corresponding operator system coproduct (see, {\it e.g.}, \cite{kavruk2011}).
By \cite{fkpt} or since each generator $u_v$ is in the span of $e_{v,i}, 1 \le i \le c,$ $\cl S(n,c)$ can be identified with the span of the generators of the group $F(n,c)$
inside the C*-algebra $C^*(F(n,c))$.
Then
\[
  \cl S(n,c) = {\rm span}\{e_{v,i} : v\in V, 1 \le i \le c\}.
\]

We note the relations
\begin{equation}\label{eq_rel}
e_{v,j}^2 = e_{v,j} = e_{v,j}^*, \ \ \ e_{v,i}e_{v,j}=0, i \ne j, \ \ \ \sum_{i=1}^c e_{v,i} = 1, \ v\in V, 1\leq j\leq c.
\end{equation}

Because the left regular representation of $F(n,c)$ is faithful on the group algebra, the C*-algebra $C^*(F(n,c))$ can thus be viewed as the universal C*-algebra generated by
the set $\cl E = \{e_{v,i} : v\in V, 1\leq i \leq c\}$ satisfying (\ref{eq_rel}). 

A \emph{word in $\cl E$} is an element of the form
\begin{equation}\label{eq_word}
\alpha = e_{v_1,i_1}e_{v_2,i_2}\cdots e_{v_k,i_k}
\end{equation}
where $v_j\in V$ and $1\leq i_j\leq c$, $j = 1,\dots,k$.
The length $|\alpha|$ of a word $\alpha$ is the smallest $k$ for which $w$ can be written in the form (\ref{eq_word}).
A \emph{polynomial of degree $k$} is an element $p$ of $C^*(F(n,c))$ of the form
$p = \sum_{j=1}^m \lambda_j \alpha_j$, where $\alpha_j$ are words, $\lambda_j\in \bb{C}$,
and ${\rm deg }(p) \stackrel{def}{=}\max_{j=1,\dots,m} |\alpha_j| = k$.

For a given $N\in \bb{N}$, let
\[
  \cl P_N = {\rm span}\{p : \mbox{ a polynomial with } {\rm deg}(p)\leq N\},
\]
and
\[
  \cl P = {\rm span}\{\cl P_N : N\in \bb{N}\}.
\]
We note that $\cl P_N$ is an operator subsystem of $C^*(F(n,c))$ and
that $\cl P_N\subseteq \cl P_{N+1}$, $N\in \bb{N}$.
We also note that $\cl P$ is a (dense) *-subalgebra of $C^*(F(n,c))$ and
hence possesses a canonical induced operator system structure.

Note that
\[
  C^*(F(n,c))\otimes_{\max} C^*(F(n,c)) = C^*(F(n,c)\times F(n,c)),
\]
up to a (canonical) *-isomorphism.
Let
\[
  \cl A = \cl P\otimes \cl P \subseteq C^*(F(n,c)\times F(n,c)).
\]
We have that $C^*(F(n,c)\times F(n,c))$ is the universal C*-algebra
generated by two families of elements,
$\cl E = \{e_{v,i} : v\in V, 1\leq i \leq c\}$ and $\cl F = \{f_{w,j} : w\in V, 1\leq j \leq c\}$,
each of which satisfies relations (\ref{eq_rel}), as well as the commutativity relations
\begin{equation}\label{eq_commut}
e_{v,i}f_{w,j} = f_{w,j}e_{v,i}, \ \ \ v,w\in V, 1\leq i,j\leq c.
\end{equation}
We define polynomials on $\cl E\cup\cl F$ in a similar fashion -- note that,
due to (\ref{eq_commut}), each such
polynomial is a sum of products of the form $\alpha \beta$, where $\alpha$ (resp. $\beta$) is a word on $\cl E$
(resp. $\cl F$).

Let
\[
  \Gamma_N = \{\gamma : \mbox{ a word in }\cl E\cup \cl F, |\gamma|\leq N\},
\]
$\cl A_N = {\rm span} \, \Gamma_N$ and
\[
  S_N = S_{N,n,c} = \{s : \cl A\to \bb{C} : s(1) = 1, s(p^*p)\geq 0 \mbox{ for all } p\in \cl A_N\}.
\]
The functionals $s$ above, as well as all functionals appearing
hereafter, are assumed to be linear.
Set $\Gamma = \cup_{N=1}^{\infty} \Gamma_N$.
Note that
\[
  \cl A = {\rm span} \Gamma = {\rm span}\{\cl A_{N} : N\in \bb{N}\}.
\]



\begin{lemma}\label{l_SN} 
Let $s: \cl A \to \bb C$ be a linear functional.
Then $s\in S_N$ if and only if
$(s(\beta^*\alpha))_{\alpha,\beta\in \Gamma_{N}}$ is a positive matrix.
\end{lemma}
\begin{proof}
By definition, $s\in S_N$ if and only if
$s(p^*p)\geq 0$ for all $p\in \cl A_N$, if and only if
\[
  s\left(\left(\sum_{k=1}^m \overline{\lambda_k} \alpha_k^*\right)
               \left(\sum_{k=1}^m \lambda_k \alpha_k\right)\right) \geq 0,
\]
for all $\alpha_k\in \Gamma_N$, $\lambda_k\in \bb{C}$, $k = 1,\dots,m$,
if and only if
\[
  \sum_{k=1}^m \overline{\lambda_k} \lambda_l\alpha_k^* \alpha_l \geq 0,
\]
for all $\alpha_k\in \Gamma_N$, $\lambda_k\in \bb{C}$, $k = 1,\dots,m$,
if and only if the matrix $(s(\beta^*\alpha))_{\alpha,\beta\in \Gamma_N}$ is positive.
\end{proof}

Every functional on $\cl A$ can, after restriction, be considered as a
functional on 
$\cl A_N$.
Letting
\[
  S = S_{n,c} = \{s : \cl A\to \bb{C} : s(1) = 1, s(p^*p)\geq 0 \mbox{ for all } p\in \cl A\},
\]
we clearly have
\[
  S = \{s : \cl A\to \bb{C} : s|_{\cl A_N}\in S_N,  \mbox{ for all }  N\in \bb{N}\}.
\]

\begin{lemma}\label{l_GNS}
$s\in S$ if and only if there exists a state $\tilde{s}$ of the C*-algebra $\cl C = C^*(F(n,c)\times F(n,c))$
with $\tilde{s}|_{\cl A} = s$.
\end{lemma}
\begin{proof}
The statement follows from a standard construction of GNS type; the detailed arguments are omitted.
\end{proof}

Let
$\cl R$ be the ideal of all polynomials on the set of non-commuting variables $\cl E\cup\cl F$
generated by the elements
\[
  e_{v,j}^2 - e_{v,j}, \ e_{v,j}^* - e_{v,j}, \ \sum_{k=1}^c e_{v,k} - 1, \
  e_{v,i}f_{w,j} - f_{w,j}e_{v,i},
\]
where $v,w\in V, 1\leq i,j\leq c$.
Set $\cl R_N = \cl R\cap \cl A_N$.
For example, $\cl R$ contains the elements $\beta_1^*\alpha_1 - \beta_2^*\alpha_2$ if
$\beta_1^*\alpha_1 = \beta_2^*\alpha_2$ in $\Gamma$.
We let
\[
  \cl M = \left\{ (c_{\alpha,\beta})_{\alpha,\beta\in \Gamma} :
     \sum_{k=1}^m \lambda_k \beta_k^*\alpha_k \in \cl R \Longrightarrow
     \sum_{k=1}^m \lambda_k c_{\alpha_k,\beta_k} = 0 \right\}.
\]
We also let
$\cl M_N$ be the set of all compressions of matrices in $\cl M$ to
$\Gamma_N\times\Gamma_N$.
In particular, if $(c_{\alpha,\beta})_{\alpha,\beta\in \Gamma_N}$
then $c_{\alpha_1,\beta_1} = c_{\alpha_2,\beta_2}$
whenever $\beta_1^*\alpha_1 = \beta_2^*\alpha_2$.

\begin{lemma}\label{l_matrf}
There is a bijective correspondence between $\cl M$ and the set of all
linear functionals on $\cl A$, sending an element $(c_{\alpha,\beta})_{\alpha,\beta\in \Gamma}$ of $\cl M$
to the functional $f : \cl A\to \bb{C}$ given by
$f(\beta^*\alpha) = c_{\alpha,\beta}$, $\alpha,\beta\in \Gamma$.
\end{lemma}
\begin{proof}
The only thing that needs to be checked is that, given
$(c_{\alpha,\beta})_{\alpha,\beta\in \Gamma}\in \cl M$, the mapping defined on
the generators of $\cl A$ by
$f(\beta^*\alpha) = c_{\alpha,\beta}$, $\alpha,\beta\in \Gamma$, and extended by linearity,
is well-defined. This follows from the definition of $\cl M$.
\end{proof}

For $N\in \bb{N}$ we now let
\[
  C^N(n,c) = \{(s(e_{v,i}f_{w,j}))_{v,i,w,j} : s\in S_N\}.
\]
Note that 
\[
  C_{\cc}(n,c) = \{(s(e_{v,i}\otimes f_{w,j}))_{v,i,w,j} : s\in S(\cl S(n,c)\otimes_{\rm c}\cl S(n,c))\},
\]
the set of all relativistic quantum correlations.

\begin{theorem}\label{th_compl}
(i) \ A matrix $A = (a_{v,i,w,j})_{v,i,w,j}$ belongs to $C_{\cc}(n,c)$ if and only if
it can be completed to a positive matrix $B = (b_{\alpha,\beta})_{\alpha,\beta\in \Gamma} \in \cl M$
(meaning that every finite submatrix of $B$ is positive) with $b_{1,1} = 1$.

(ii) A matrix $A = (a_{v,i,w,j})_{v,i,w,j}$ belongs to $C^N(n,c)$ if and only if
it can be completed to a positive matrix $B = (b_{\alpha,\beta})_{\alpha,\beta\in \Gamma_N}\in \cl M_N$
with $b_{1,1} = 1$.
\end{theorem}
\begin{proof}
(i) Let $s\in S(\cl S(n,c)\otimes_{\rm c}\cl S(n,c))$ be a state such that
$A = (s(e_{v,i}\otimes f_{w,j}))_{v,i,w,j}$.
Let $H$ be a Hilbert space, $\pi : C^*(F(n,c))\otimes_{\max} C^*(F(n,c)) \to \cl B(H)$ be
a *-representation and $\xi\in H$ be a unit vector such that
$s(T) = \langle \pi(T)\xi,\xi\rangle$, $T\in C^*(F(n,c))\otimes_{\max} C^*(F(n,c))$.
Thus,
\[
  s(\beta^*\alpha) = \langle \pi(\alpha)\xi,\pi(\beta)\xi\rangle,
\]
and every finite submatrix of the matrix $(s(\beta^*\alpha))_{\alpha,\beta\in \Gamma}$
is positive.

Conversely, suppose that $B = (b_{\alpha,\beta})_{\alpha,\beta\in \Gamma}\in \cl M$ 
is a completion of $A$ that has
the property that all of its finite submatrices are positive.
It follows from Lemma \ref{l_SN} that the linear functional $s : \cl A\to \bb{C}$ given by
$s(\alpha) = b_{\alpha,1}$ (and well-defined by Lemma \ref{l_matrf})
is positive. Note that $s(1) = 1$. By Lemma \ref{l_GNS},
$s$ is the restriction to $\cl A$ of a state of $C^*(F(n,c)\times F(n,c))$.
Thus, $A\in C_{\cc}(n,c)$.

(ii) Suppose that $A = (a_{v,i,w,j})_{v,i,w,j}$ belongs to $C^N(n,c)$. 
Then there exists $s\in S_N$ such that $a_{v,i,w,j} = s(e_{v,i}f_{w,j})$, $v,w\in V$, $i,j = 1,\dots,c$. 
By Lemma \ref{l_SN}, $(s(\beta^*\alpha))_{\alpha,\beta\in \Gamma_N}$ is a 
positive completion of $A$ that lies in $\cl M_N$. 

Conversely, suppose that $B$ is a positive completion of $A$ that lies in $\cl M_N$. 
By Lemma \ref{l_matrf}, $B$ is the compression to $\Gamma_N\times\Gamma_N$ 
of a matrix of the form $(s(\beta^*\alpha))_{\alpha,\beta\in \Gamma}$ and, by 
Lemma \ref{l_SN}, $s\in S_N$. It follows that $A\in A_N$. 
\end{proof}

The following corollary follows directly from Lemma \ref{l_matrf}; we omit the detailed proof. 

\begin{corollary}\label{c_qcita}
Let $G = (V,E)$ be a graph. We have that $\chi_{\qc}(G)\leq c$ if and only if
there exists $s\in S$ such that
\begin{align}
  \label{cc}
  \forall v, \forall i \ne j, s(e_{v,i}f_{v,j})  &=0, \nonumber\\
  \forall (v,w) \in E, \forall i, s (e_{v,i}f_{w,i} ) &=0.
\end{align}
\end{corollary}

\begin{lemma}\label{l_bound}
Let $s\in S_N$. Then $|s(\gamma)|\leq 1$ for all $\gamma\in \Gamma$ with $|\gamma|\leq 2N$.
\end{lemma}
\begin{proof}
We first show that $0\leq s(\gamma^*\gamma)\leq 1$ for all words $\gamma$ with $|\gamma|\leq N$.
To this end, we use induction on $|\gamma|$.
We have that $s(e_{v,i}) =  s(e_{v,i}^*e_{v,i}) \geq 0$ while
$\sum_{i=1}^n s(e_{v,i}) = 1$; it follows that $0\leq s(e_{v,i})\leq 1$ for all $v,i$.
By symmetry, $0\leq s(f_{w,j})\leq 1$ for all $w,j$, that is, the claim holds for all
$\gamma$ with $|\gamma| = 1$.

Suppose $|\gamma|\leq k$, for some $k\leq N-1$.
If $v\in V$ then
\[
  0\leq s((e_{v,i}\gamma)^*(e_{v,i}\gamma)) = s(\gamma^*e_{v,i}\gamma) \ 
    \mbox{ and } \ \sum_{k=1}^c s(\gamma^*e_{v,k}\gamma) = s(\gamma^*\gamma)\leq 1.
\]
It follows by induction that
$0\leq s(\gamma^*\gamma)\leq 1$ whenever $|\gamma|\leq N$.

Now suppose $|\gamma|\leq 2N$ and write $\gamma = \beta^*\alpha$ for some
words $\alpha$ and $\beta$ of length at most $N$.
By Lemma \ref{l_SN}, the matrix
$\left(\smallmatrix s(\alpha^*\alpha) & s(\beta^*\alpha)\\
s(\alpha^*\beta) & s(\beta^*\beta)\endsmallmatrix \right)$ is positive.
Thus, 
\[
  |s(\gamma)| = |s(\beta^*\alpha)|\leq \max\{s(\alpha^*\alpha),s(\beta^*\beta)\} \leq 1.
\]
\end{proof}

\begin{lemma}\label{l_compact}
Let $s_N\in S_N$, $N\in \bb{N}$. There exists a subsequence $(s_{N'})_{N'}$ of $(s_N)_{N}$
which converges pointwise to an element of $S$.
\end{lemma}
\begin{proof}
Since $s_N\in S_1$ for all $N$, Lemma \ref{l_bound} implies that
$|s_N(\gamma)|\leq 1$ whenever $\gamma\in \Gamma_2$.
Thus, there exists a subsequence
\[
  s_{N_1}^{(1)}, s_{N_2}^{(1)}, s_{N_3}^{(1)}, \ldots
\]
such that $(s_{N_k}^{(1)}(\gamma))_{k\in \bb{N}}$ converges whenever $\gamma\in \Gamma_2$.
Deleting the first terms of this subsequence, if necessary, we may assume that
$s_{N_k}^{(1)}\in S_2$ for all $k$.

Using Lemma \ref{l_bound}, it now follows that there exists a subsequence $(s_{N_k}^{(2)})_{k\in \bb{N}}$
of $(s_{N_k}^{(1)})_{k\in \bb{N}}$ such that
$(s_{N_k}^{(2)}(\gamma))_{k\in \bb{N}}$ converges whenever $\gamma\in \Gamma_4$.
Deleting the first terms of this subsequence, if necessary, we may assume that
$s_{N_k}^{(2)}\in S_3$ for all $k$.

Continuing inductively, for each $m$ we obtain a sequence
$(s_{N_k}^{(m)})_{k\in \bb{N}}$, which is a subsequence of
$(s_{N_k}^{(m-1)})_{k\in \bb{N}}$, and which has the property that
$(s_{N_k}^{(m)}(\gamma))_{k\in \bb{N}}$ converges whenever $\gamma\in \Gamma_{2m}$.

We claim that
$(s_{N_k}^{(k)}(\gamma))_{k\in \bb{N}}$ converges whenever $\gamma\in \Gamma$.
To see this, let $\gamma\in \Gamma$ and assume that $|\gamma|\leq 2m$ for some $m$.
Then $(s_{N_k}^{(k)}(\gamma))_{k\geq m}$ is a subsequence of
$(s_{N_k}^{(m)}(\gamma))_{k\in \bb{N}}$ and hence converges.

Set $s(p) = \lim_{k\to\infty} s_{N_k}^{(k)}(\gamma)$, $p\in \cl A$.
Since $s$ is a pointwise limits of linear maps, it is a linear map itself. Clearly, $s(1) = 1$.
We claim that $s\in S$. If $p\in \cl A$ is a polynomial, then $p\in \cl A_N$ for some $N$
and hence $s_{N_k}^{(k)}(p^*p) \geq 0$ whenever $N_k\geq N$. It follows that $s(p^*p)\geq 0$.
\end{proof}

\begin{theorem}\label{th_inte}
We have that $\cap_{N=2}^{\infty} C^N(n,c) = C_{\cc}(n,c)$.
\end{theorem}
\begin{proof}
Suppose that $s\in S(\cl S(n,c)\otimes_{\rm c}\cl S(n,c))$.
By \cite[Lemma 2.8]{pt_chrom},
\begin{equation}\label{eq_coicmax}
\cl S(n,c)\otimes_{\rm c}\cl S(n,c))\subseteq_{\rm coi} C^*(F(n,c))\otimes_{\max} C^*(F(n,c)),
\end{equation}
and hence $s$ extends to a state $\tilde{s}$ on $C^*(F(n,c))\otimes_{\max} C^*(F(n,c))$.
Letting $s_N = \tilde{s}|_{\cl A_N}$, we have that $s_N\in S_N$;
clearly, $s(e_{v,i}\otimes f_{w,j}) = s_N(e_{v,i}f_{w,j})$, $N\geq 2$, so $C_{\cc}(n,c)\subseteq \cap_{N=2}^{\infty} C^N(n,c)$.

Conversely, suppose that $A\in \cap_{N=2}^{\infty} C^N(n,c)$.
For each $N\geq 2$, let
$s_N\in S_N$ be such that
$A = (s_N(e_{v,i}f_{w,j}))_{v,i,w,j}$.
By Lemma \ref{l_compact}, there exists a subsequence of $(s_N)$
which converges pointwise to an element $s\in S$.
Since $s_N(e_{v,i}f_{w,j}) = s_M(e_{v,i}f_{w,j})$ for all $N,M\in \bb{N}$,
we have that $s(e_{v,i}f_{w,j}) = s_N(e_{v,i}f_{w,j})$, $N\in \bb{N}$.
By Lemma \ref{l_GNS}, $s$ is the restriction of a state $\tilde{s}$ on 
$C^*(F(n,c)\times F(n,c)) = C^*(F(n,c))\otimes_{\max} C^*(F(n,c))$. 
By (\ref{eq_coicmax}), we have that $A\in C_{\cc}(n,c)$.
\end{proof}

\begin{definition}\label{d_Ncol}
Let $G$ be a graph on a set $V$ of $n$ vertices with set $E$ of edges.
A \emph{quantum $N,c$-colouring} of $G$ is a state $s\in S_{N,n,c}$ such that
\begin{align}
\forall v, \forall i \ne j, s(e_{v,i}f_{v,j})  &=0, \nonumber\\
\forall (v,w) \in E, \forall i, s (e_{v,i}f_{w,i} ) &=0.
\end{align}
The \emph{$N$th quantum chromatic number} $\chi_{\qc}^N(G)$ of $G$ is the smallest positive integer $c$
for which there exists a quantum $N,c$-colouring of $G$.
\end{definition}

According to Theorem \ref{th_compl}, quantum $N,c$-colourings of $G$ correspond bijectively to
matrices $A = (a_{v,i,w,j})_{v,i,w,j}\in \cl C^N(n,c)$ such that
\[
  a_{v,i,v,j} =0, \ \forall v, \forall i \ne j \ 
    \mbox{ and } \ a_{v,i,w,i} = 0, \forall (v,w) \in E, \forall i.
\]

\begin{theorem}\label{th_ine}
Let $G$ be a graph.
We have $\chi_{\qc}^N(G)\leq \chi_{\qc}^{N+1}(G)\leq \chi_{\qc}(G)$, for every $N\in \bb{N}$.
Moreover, $\lim_{N\to\infty} \chi_{\qc}^N(G) = \chi_{\qc}(G)$.

Thus, given $n\in \bb{N}$, there exists $N\in \bb{N}$ such that
$\chi_{\qc}(G) = \chi_{\qc}^N(G)$ for all graphs on at most $n$ vertices.
\end{theorem}
\begin{proof}
Let $c = \chi_{\qc}^{N+1}(G)$ and $s\in S_{N+1,n,c}$
be a quantum $N,c$-colouring of $G$. Then $s\in S_{N,n,c}$.
Since $\chi_{\qc}^N(G)$ is the minimum of all $c$ for which there exists a
quantum $N,c$-colouring for $G$, we have that
$\chi_{\qc}^N(G)\leq \chi_{\qc}^{N+1}(G)$.

A similar argument, using the fact that $S\subseteq S_{N,n,c}$, shows the second inequality.
It follows that the sequence $(\chi_{\qc}^N(G))_{N=1}^{\infty}$ stabilises, that is,
there exist $c,N_0\in \bb{N}$ such that $\chi_{\qc}^N(G) = c$ for all $N\geq N_0$.
For each $N\geq N_0$, let $s_N\in S_N$ be a quantum $N,c$-colouring of $G$.
By Lemma \ref{l_compact} and Corollary \ref{c_qcita}, $\chi_{\qc}(G)\leq c$.

Finally, since the sequence stabilises for each graph and there are only 
finitely many graphs on at most $n$ vertices, there exists $N \in \bb N$ such that 
$\chi_{\qc}^N(G) = \chi_{\qc}(G)$, for all graphs on at most $n$ vertices.
\end{proof}

\begin{remark}\label{Nndefn}  Let $N_n$ be the least integer such that $\chi_{\qc}^{N_n}(G) = \chi_{\qc}(G)$ for all graphs on at most $n$ vertices. Because we obtain the integer $N_n$ by a 
compactness argument, we do not have effective bounds on $N_n$.  
In particular, we do not know if $\sup_n \, N_n$ is bounded or have any information on its growth rate.
\end{remark}

\section{An SDP for the commuting quantum chromatic number}
\label{s_sdp}
In this section we prove that for each graph $G=(V,E)$ determining if 
$\chi_{\qc}(G) \le c$ is decidable by a semidefinite programming problem.

We fix a graph $G = (V,E)$, let $n = |V|$ and let $N := N_n$ be the natural number given by Remark~\ref{Nndefn}. Let $c\in \bb{N}$.  
We set $C^N= C^N(n,c) = \{ (s(e_{v,i}f_{w,j}): s \in S_{N} \} \subseteq M_{nc}$.

Recall that a set is called a \emph{spectrahedron} if it can be realised as the 
intersection of the set of positive semidefinite matrices of some size with an affine subset.

Recall that the graph functional $L_{G,c}: M_{nc} \to \bb C$ is defined by
\[ 
  L_{G,c}\big( (a_{v,i,w,j}) \big) = \sum_{i \ne j, v} a_{v,i,v,j} + \sum_{i, v\sim w} a_{v,i,w,i}.
\]

When the value of $c$ is understood, we will often write $L_G$ for $L_{G,c}$.

\begin{prop} For any $N \in \bb N$, the problem of minimising 
$L_{G,c}$ over the set $C^N$ is an SDP.
\end{prop}
\begin{proof}
 Recall that $\cl M_N$ is a space of  $|\Gamma_N| \times |\Gamma_N|$ matrices defined by some linear constraints. Thus, the set $\cl B_N$ of positive semidefinite matrices $B \in \cl M_N$ with $b_{1,1}=1$ is a spectrahedron.  By Theorem~\ref{th_compl} (ii), every element of $C^N$ is the restriction of such a matrix $B$ in $\cl M_N$  to 
some of its components. Thus, if we extend $L_{G,c}$ to a linear functional $F$ on $\cl B_N$ by setting it equal to 0 on all the components not included in $C^N$ then we see that minimizing $F$ over $\cl B_N$ is the same as minimizing $L_{G,c}$ over $C^N$. 
\end{proof}

\begin{thm} Let $G=(V,E)$ be a graph on $n$ vertices and let 
$N= max \{ N_n, 2\}$, where $N_n$ is defined in Remark~\ref{Nndefn}. 
Then $\chi_{\qc}(G) \le c$ if and only if  $\inf \{L_{G,c}(A): A \in C^N(n,c)\} = 0$.  
Hence, the problem of determining if $\chi_{\qc}(G) \le c$ is solvable by this SDP.
\end{thm}
\begin{proof}  
For any $s \in S_N$, any $v,w \in V$,  any $1 \le i,j \le c$, since $N \ge 2$, we have that
\[ 0 \le s( (e_{v,i}f_{w,j})^*(e_{v,i}f_{w,j})) = s(e_{v,i}f_{w,j}).\]
Thus, all the elements of $C^N$ are non-negative and, consequently, $L_{G,c}(A) \ge 0$ on $C^N$. 

Since $C^N$ is a compact set the infimum is 0 if and only if it is attained at 
some matrix $A\in C^N$, but in that case we have that $A$ is the image of a 
state that defines a quantum $N,c$-colouring  and
so $\chi_{\qc}(G) = \chi_{\qc}^N(G) \le c$. 
\end{proof}

Note that since $C^N$ is a compact set, the above infimum is actually a minimum.

\begin{remark}  It is known that computing $\chi_{\q}(G)$ is an NP-hard problem \cite{ji}, 
but it is not known if computing $\chi_{\qc}(G)$ or $\chi_{\qa}(G)$ is NP-hard. A proof that did not rely on Tsirelson's or Connes' conjecture that these are also NP-hard would be interesting. A proof that either of these is of complexity P would be a dramatic result. It would show that either the corresponding conjecture is false or that P=NP.
\end{remark}

We can strengthen the above result a bit as follows.

\begin{thm} 
For each $n\in \bb{N}$, there is a constant $\epsilon_n>0$ such that if 
$G=(V,E)$ is a graph on $n$ vertices and $N= \max \{ 2, N_n \}$, then $\chi_{\qc}(G) \le c$ if and only if
$\inf \{ L_{G,c}(A): A \in C^N \} < \epsilon_n$. 
\end{thm}
\begin{proof}
The graphs on $n$ vertices split into two subsets: 
those for which $\chi_{\qc}(G) \le c$, and those for which $\chi_{\qc}(G) > c$.   
For each graph in the latter set, we have that $\inf \{ L_{G,c}(A): A \in C^N \} = b_G >0$.   
Since there are only finitely many such graphs, we may set $\epsilon_n = \min \{ b_G \}$ 
over this set of graphs.
\end{proof}

The above result is not of much computational use without 
estimates for $\epsilon_n$, but it might be of theoretical use.  
Since we know that we only need to get the SDP within $\epsilon_n$ 
this may give us a crude operation/complexity count.

\section{Synchronous states and the tracial rank}
\label{s_tre} 
D.\ Roberson and L.\ Man\v{c}inska~\cite{arxiv:1212.1724} introduced 
the projective rank $\xi_{\f}(G)$ of a graph $G$ 
and showed that $\xi_{\f}(G) \le \chi_{\q}(G)$.  
This lower bound has been crucial for computing the quantum chromatic numbers of some graphs.  
Unfortunately, the proof of this estimate uses in a critical way the fact that 
the involved representations are finite dimensional and, 
consequently, we do not know if it is also a lower bound for $\chi_{\qc}(G)$.  To ameliorate this situation, we will develop an analogous quantity that is better suited to work with infinite dimensional representations.  

Our first result applies to a larger family of games than the graph colouring game, namely games 
where Alice and Bob have the same set of inputs $V$ and require that when Alice and Bob receive the same input, then they must produce the same output. 
Given a correlation $(p(i,j|v,w))_{v,i,w,j}$, set
$$p(i = j|v,w) := \sum_{i=1}^c p(i,i|v,w).$$ 
Then a perfect strategy for this game means 
$p(i=j|v=w) =1$, that is, $p(i = j|v,v) = 1$ for all $v\in V$.
Our result can be summarised as saying that such correlations always arise from tracial states. 
Recall that a state $s$ on a C*-algebra $\cl A$ 
is called tracial provided that $s(xy) = s(yx)$ for all $x,y\in \cl A$. 
We show that, in fact, the projective rank of a graph can be described
by using tracial states on finite dimensional C*-algebras.

\begin{defn} For ${\rm x} \in \{{\rm loc},  {\rm q}, {\rm qa}, {\rm qc}\}$, 
we call a correlation $(p(i,j|v,w))$ from $C_{\rm x}(n,c)$ \emph{synchronous} 
if it satisfies the condition $p(i=j| v=w) =1$, and let 
$C^s_{\rm x}(n,c)\subseteq C_{\rm x}(n,c)$  
denote the subset of all synchronous correlations. 
\end{defn}

Note that $C_{\rm q}^s(n,c) = C^s_{\cc}(n,c) \cap C_{\rm q}(n,c)$. 

\begin{defn}\label{d_real}
A \emph{realisation} of an element $(p(i,j|v,w))_{v,i,w,j}$ of 
$C_{\cc}(n,c)$ is a tuple
$\Big(\big((E_{v,i})_{i=1}^c\big)_{v \in V}, \big((F_{w,j})_{j=1}^c \big)_{w \in V},\cl H,\eta \Big)$, where      
$V$ is an index set of cardinality $n$,
$\cl H$ is a Hilbert space, 
$\eta \in \cl H$ is a unit vector, and  $E_{v,i}, F_{w,j} \in \cl B(\cl H)$ are projections satisfying
\begin{itemize}
\item[(i)] $p(i,j|v,w) = \langle E_{v,i}F_{w,j} \eta, \eta \rangle, \ \ v,w\in V,  \ i,j = 1,\dots,c$;
\item[(ii)] $\sum_{i=1}^c E_{v,i} = \sum_{j=1}^c F_{w,j} = I, \ \ v,w\in V$;
\item[(iii)] $E_{v,i}F_{w,j} = F_{w,j} E_{v,i}, \ \ v,w\in V,  \ i,j = 1,\dots,c$. 
\end{itemize}
\end{defn}

When $n$ and $c$ are understood, to avoid excessive notation, we will often denote a realisation by simply  $\big( (E_{v,i}), (F_{w,j}), \cl H, \eta \big)$.

\begin{thm}\label{findimrealisation} A correlation $((p(i,j|v,w))_{v,i,w,j}$ belongs to $C_{\rm q}(n,c)$
if and only if it has a realisation
$\big((E_{v,i}), (F_{w,j}), \cl H,\eta \big)$ for which $\cl H$ is 
finite dimensional. 
\end{thm}
\begin{proof} 
Suppose that $((p(i,j|v,w))_{v,i,w,j}$ is a correlation which possesses 
a realisation $\big((E_{v,i}), (F_{w,j}), \cl H,\eta \big)$ for which $\cl H$ is 
finite dimensional.
We now essentially recall the argument from
the unpublished preprint \cite[Theorem~1]{SW} which allows us to pass to spacial tensoring.
Let $\cl C$ (resp. $\cl D$) be the C*-algebra generated by $\{E_{v,i} : v\in V, i = 1,\dots,c\}$
(resp. $\{F_{w,j} : w\in V, j = 1,\dots,c\}$). 
Since $\cl C$ is finite dimensional, we may assume, without loss of generality,
that $H = \oplus_{s=1}^t \bb{C}^{k_s}\otimes \bb{C}^{l_s}$ and 
$\cl C = \oplus_{s=1}^t M_{k_s}\otimes 1_{l_s}$. 
Since $\cl D$ and $\cl C$ commute, 
we have that $\cl D$ is contained in the C*-algebra 
$\oplus_{s=1}^t 1_{k_s}\otimes M_{l_s}$. 
Thus, $E_{v,i} = \oplus_{s=1}^t E^{s}_{v,i}\otimes I_{l_s}$ and 
$F_{w,j} = \oplus_{s=1}^t I_{k_s}\otimes F_{w,j}^{s}$ for some projections 
$E^{s}_{v,i}\in M_{k_s}$, $F_{w,j}^{s}\in M_{l_s}$. 
Now let $k = \max\{k_1,\dots,k_t\}$ and $l = \max\{l_1,\dots,l_t\}$. 
Consider the Hilbert space $H$ as a subspace, in the natural way, 
of $\oplus_{s=1}^t \bb{C}^{k}\otimes \bb{C}^{l}$, and identify the latter space 
with $\bb{C}^{k}\otimes \bb{C}^{lt}$. Under these identifications, 
the projections $E_{v,i}$ (resp. $F_{w,j}$) have the form
$E_{v,i} = E_{v,i}'\otimes I_{lt}$ (resp. $F_{w,j} = I_{k}\otimes F_{w,j}'$), 
for some projections $E_{v,i}'$ (resp. $F_{w,j}'$) on $\bb{C}^{k}$
(resp. $\bb{C}^{lt}$). 
It follows that $((p(i,j|v,w))_{v,i,w,j}\in C_{\rm q}(n,c)$.

Conversely, assume that $(p(i,j|v,w)) \in C_{\rm q}(n,c)$.  
Then there exist finite dimensional Hilbert spaces $\cl H_A$ and $\cl H_B$,   
POVM's  $\big( (P_{v,i})_{i=1}^c \big)_{v \in V}$ on $\cl H_A$,
POVM's $\big( (R_{w,j})_{j=1}^c \big)_{w \in V}$ on $\cl H_B$ 
and a unit vector $\eta \in \cl H_A \otimes \cl H_B$ such that 
$p(i,j|v,w) = \langle  P_{v,i} \otimes R_{w,j} \eta, \eta \rangle.$

For convenience, set $V= \{1,\dots,n \}$. Let $\tilde{\cl H}_A = \cl H_A \otimes \bb C^c$, 
regarded as the direct sum of $c$ copies of $\cl H_A$.
Note that $\tilde{\cl H}_A$ is still finite dimensional and define an inclusion  
$W: \cl H_A \to \tilde{\cl H}_A$ via $h \to \big( P_{1,1}^{1/2}h,\dots, P_{1,c}^{1/2}h \big).$ 
The fact that $( P_{1,i} )_{i=1}^c$ is a POVM implies that this inclusion is an isometry. 

Define operators on $\tilde{\cl H}_A$ by setting:
$\tilde{P}_{1,i} = I_{\cl H_A} \otimes E_{i,i}$ 
(where $E_{i,i}$ here denotes the corresponding diagonal matrix unit on $\bb{C}^c$) 
and for $v \ne 1,$ let 
$\tilde{P}_{v,i}$ be the operator matrix, with $(k,l)$-entry, 
\[ \tilde{P}_{v,i} = \big( P_{1,k}^{1/2} P_{v,i} P_{1,l}^{1/2} \big) , i \ne 1,\]
and 
\[ \tilde{P}_{v,1} = \big( P_{1,k}^{1/2} P_{v,1} P_{1,l}^{1/2} \big) +(I_{\tilde{\cl H}} - WW^*). \]
Note that this standard dilation trick turns the POVM $(P_{1,i})_{i=1}^c$ into a 
PVM $(\tilde{P}_{1,i})_{i=1}^c$, and turns each 
POVM $(P_{v,i})_{i=1}^c$, $v\neq 1$, into a new POVM $(\tilde{P}_{v,i})_{i=1}^c$ on the larger space. 
Moreover, $W^*\tilde{P}_{v,i} W = P_{v,i}.$

Also, note that for $i \ne 1$,
\[ \tilde{P}_{v,i}^2 = \sum_{t=1}^c  \big( P_{1,k}^{1/2} P_{v,i} P_{1,t}^{1/2}P_{1,t}^{1/2} P_{v,i} P_{1,l}^{1/2} \big) 
= \big( P_{1.k}^{1/2} P_{v,i}^2 P_{1,l}^{1/2} \big).\]  Similarly, 
\[ \tilde{P}_{v,1}^2 = \big( P_{1,k}^{1/2} P_{v,1}^2 P_{1,l}^{1/2} \big) + (I_{\tilde{\cl H}} - WW^*).\]
Thus, if $P_{v,i}$ is a projection, then it's dilation $\tilde{P}_{v,i}$ is also a projection.  

Hence, if any $\big( P_{v,i} \big)_{i=1}^c$ is already a PVM, that property is preserved by the dilation.
It follows that if we repeat this standard dilation trick $n$ times, 
once for each $v$, then we will obtain a family of PVM's $\big( \hat{P}_{v,i} \big)_{i=1}^c$, $1 \le v \le n,$ on the finite dimensional Hilbert space $\hat{\cl H}_A =\cl H_A \otimes \bb C^{2^nc}$ and an isometric embedding  $W_A:\cl H_A \to \hat{H}_A$ such that
$W_A^*\hat{P}_{v,i} W_A = P_{v,i}.$ 

Repeating the same process for the POVM's $\big( R_{w,j} \big)$ on $\cl H_B,$ we obtain a family of PVM's $\big( \hat{R}_{w,j} \big)$ on a finite dimensional space $\hat{\cl H}_B$ and an isometry $W_B: \cl H_B \to \hat{\cl H}_B$ such that $R_{w,j}= W_B^* \hat{R}_{w,j} W_B.$

Finally,  $\langle (\hat{P}_{v,i} \otimes I)(I \otimes \hat{R}_{w,j}) (W_A \otimes W_B) \eta, (W_A \otimes W_B) \eta \rangle =p(i,j|v,w)$ so that  
$$\Big( \big( \hat{P}_{v,i}\otimes I \big), \big( I \otimes \hat{R}_{w,j} \big),  \hat{\cl H}_A \otimes \hat{\cl H}_B,  (W_A \otimes W_B) \eta \Big)$$ 
is a realisation of $(p(i,j|v,w))$ by commuting PVM's on a finite dimensional Hilbert space. It now follows
that $((p(i,j|v,w))_{v,i,w,j}$ is in $C_{\rm q}(n,c)$.
\end{proof}

\begin{remark}
Similarly it can be shown that, $((p(i,j|v,w)) \in C_{\rm loc}(n,c)$ if and only if the realisation can be chosen such that all the operators commute. 
We do not know of an analogous characterisation of correlations in $C_{\rm qa}(n,c).$ 
\end{remark}

\begin{thm}\label{trthm} 
Let $(p(i,j|v,w)) \in C^s_{\cc}(n,c)$ be a
synchronous correlation with realisation 
$\{(E_{v,i})_{i=1}^c, (F_{w,j})_{j=1}^c,\cl H,\eta\}$.  Then
\begin{itemize}
\item[(i)] $E_{v,i} \eta = F_{v,i} \eta, \ \ \ v\in V, i = 1,\dots,c$;
\item[(ii)] $p(i,j|v,w) = \langle E_{v,i}E_{w,j} \eta, \eta \rangle = \langle F_{w,j}F_{v,i} \eta, \eta \rangle = p(j,i|w,v)$
\item[(iii)] The functional $s : X \to \langle X \eta, \eta \rangle$ is a tracial state 
on the C*-algebra generated by the set $\{E_{v,i} : v\in V, i = 1,\dots,c\}$
(resp. $\{F_{w,j} : w\in V, j = 1,\dots,c\}$).
\end{itemize}
Conversely, given a family of projections $\{ e_{v,i}: v \in V, \, 1 \le i \le c \}$ in a unital C*-algebra $\cl A$ such that 
$\sum_{i=1}^c e_{v,i} = I$, $v\in V$, 
and a tracial state $s$ on $\cl A$, then 
$(p(i,j|v,w)) =(s(e_{v,i} e_{w,j}))_{v,i,w,j}$ is in $C_{\cc}^s(n,c)$. That is,  
there exists a Hilbert space $\cl H$, a unit vector $\eta\in \cl H$ and  
mutually commuting POVM's $(E_{v,i})_{i=1}^c$ and $(F_{w,j})_{j=1}^c$ on $\cl H$
which are a realisation of $(s(e_{v,i} e_{w,j}))_{v,i,w,j}$ additionally satisfying
\begin{equation}\label{eq_sexp}
s(e_{v,i} e_{w,j}) = \langle E_{v,i}E_{w,j} \eta, \eta
\rangle = \langle F_{w,j}F_{v,i} \eta, \eta \rangle = \langle E_{v,i}
F_{w,j} \eta, \eta \rangle.
\end{equation}
\end{thm}
\begin{proof} 
Applying the Cauchy-Schwarz inequality, for every $v\in V$, we have the following
chain of identities and inequalities.
\begin{align*}
1 &=  \sum_{i,j=1}^c p(i,j|v,v) = \sum_{i = 1}^c p(i,i|v,v) 
   = \sum_{i=1}^c \langle E_{v,i}F_{v,i} \eta, \eta \rangle\\ 
  &=  \sum_{i=1}^c \langle F_{v,i} \eta, E_{v,i} \eta \rangle 
   \le \sum_{i=1}^c \|F_{v,i} \eta\| \|E_{v,i}   \eta\| \\
  &\le \left(\sum_{i=1}^c \|F_{v,i}\eta \|^2 \right)^{1/2} \left(\sum_{i=1}^c \|E_{v,i} \eta\|^2 \right)^{1/2} \\
  &= \left(\sum_{i=1}^c \langle F_{v,i} \eta, \eta \rangle \right)^{1/2} \left( \sum_{i=1}^c \langle E_{v,i} \eta, \eta \rangle \right)^{1/2} = 1.
\end{align*}
Thus, we must have equality throughout. In particular, 
the equality between the 2nd and 3rd lines implies that the vectors 
$\big( \|F_{v,1} \eta\|,\dots,\|F_{v,c} \eta\|)$ and $(\|E_{v,1} \eta \|,\dots,\|E_{v,c} \eta \| \big)$ are equal. 
Thus, $\|F_{v,i}\eta\| = \|E_{v,i}\eta\|$, $v\in V$, $1 \le i \le c$.  
On the other hand, the equality on the second line implies
that $F_{v,i} \eta = \alpha_i E_{v,i} \eta$, for some $|\alpha_i| =1$, $i = 1,\dots,c$.
If $E_{v,i} \eta \ne 0$ then
\[\alpha_i E_{v,i} \eta = F_{v,i}^2 \eta = F_{v,i}( \alpha_i E_{v,i} \eta) = \alpha_i E_{v,i} F_{v,i} \eta = \alpha_i E_{v,i}(\alpha_i E_{v,i} \eta) = \alpha_i^2 E_{v,i} \eta,\] which forces $\alpha_i =1$.   Thus,
\begin{equation}\label{eq_sameon}
E_{v,i} \eta = F_{v,i} \eta, \ \ \ v\in V, i = 1,\dots,c.
\end{equation}
and so (i) holds.

To prove (ii), note that, by condition (i) of Definition \ref{d_real}, we have
\[ p(i,j|v,w) = \langle E_{v,i} F_{w,j} \eta, \eta \rangle = \langle E_{v,i}E_{w,j} \eta, \eta \rangle.\]
Using condition (iii) of Definition \ref{d_real}, we have
\[\langle E_{v,i} F_{w,j} \eta, \eta \rangle = \langle F_{w,j} E_{v,i} \eta, \eta \rangle 
= \langle F_{w,j} F_{v,i} \eta, \eta \rangle .\]
Finally,
\[p(j,i| w,v) = \langle E_{w,j} E_{v,i} \eta, \eta \rangle = \langle \eta, E_{w,j} E_{v,i} \eta \rangle = \langle E_{v,i} E_{w,j} \eta, \eta \rangle = p(i,j|v,w).\]

Combining (i) with commutativity we have that
\[ E_{v_1,i_1}E_{v_2,i_2} \eta = E_{v_1,i_1} F_{v_2,i_2} \eta = F_{v_2,i_2}E_{v_1.i_1} \eta = F_{v_2,i_2} F_{v_1,i_1} \eta.\] Proceeding inductively, we have the following word reversal:
\begin{equation}\label{word}
 E_{v_1,i_1}E_{v_2,i_2} \cdots E_{v_k,i_k} \eta = F_{v_k, i_k} \cdots F_{v_2,i_2} F_{v_1,i_1} \eta. \end{equation}

To prove (iii),
let $W$ be an operator that is a product of elements of the set 
$\{E_{v,i} : v\in V, i = 1,\dots,c\}$; then
\begin{align*}
  s(E_{v,i}W) &= \langle E_{v,i}W \eta, \eta \rangle 
              = \langle W\eta, E_{v,i} \eta \rangle 
              = \langle W \eta, F_{v,i} \eta \rangle \\ 
              &= \langle F_{v,i}W \eta, \eta \rangle 
              = \langle WF_{v,i} \eta, \eta \rangle 
              = \langle WE_{v,i} \eta, \eta \rangle 
              = s(WE_{v,i}).
\end{align*}
Thus, we have 
\begin{align*}
  s \big( (E_{v_1,i_1}E_{v_2,i_2})W \big) 
    &= s \big( E_{v_1,i_1} (E_{v_2,i_2} W) \big)
     = s \big( E_{v_2,i_2} (WE_{v_1,i_1}) \big) \\ 
    &= s \big( W E_{v_1,i_1} E_{v_2,i_2} \big).
\end{align*}
The general case follows by induction and the fact that the linear combinations of 
the words on the set $\{E_{v,i} : v\in V, i = 1,\dots,c\}$ are dense in 
the C*-algebra generated by this set. 

The proof that $s$ is a tracial state on the C*-algebra generated by the
set $\{F_{w,j} : w\in V, j = 1,\dots,c\}$ is identical.

Finally, assume that we have a unital C*-algebra $\cl A$, a tracial state
$s$, and projections $e_{v,i}$ as above. It is clear that
$s(e_{v,i}e_{v,j}) = 0$ whenever $v\in V$ and $i\neq j$; thus, 
$(p(i,j|v,w))_{v,i,w,j}$ is synchronous. 
Without loss of generality, we can assume that $\cl A$ is generated by the set 
$\{e_{v,i} : v\in V, i = 1,\dots,c\}$.
The GNS construction associated with $(\cl A, s)$, produces
a Hilbert space $\cl H$, a unital *-homomorphism $\pi: \cl A \to B(\cl H)$ and a unit vector $\eta\in \cl H$ 
such that $s(X) = \langle \pi(X) \eta,\eta \rangle$, $X\in \cl A$.  
Set $E_{v,i} = \pi(e_{v,i})$, $v\in V$, $i = 1,\dots,c$. Since these operators are the images of projections that sum to 1, they form a PVM. 
By construction, $\cl H$ is the cyclic subspace corresponding to $\eta$. 
Thus, every vector in $\cl H$ can be approximated by a sum of the
form
\[ \sum_{r=1}^k W_r \eta, \]
where $W_1,\dots,W_r$ are words on $\{E_{v,i} : v\in V, i = 1,\dots,c\}$.

Fix $v\in V$ and $j\in \{1,\dots,c\}$. 
Using the facts that $s$ is a tracial state and that $E_{v,j}^*= E_{v,j}$, we have that
\begin{align*}
\left\|\sum_{r=1}^k W_r E_{v,j} \eta \right\|^2  
&\le \left\| \sum_{r=1}^k W_rE_{v,j} \eta \right\|^2 +
\left\|\sum_{r=1}^k W_r(I -E_{v,j}) \eta\right\|^2\\
&\!\!\!\!\!\!\!\!\!\!\!\!\!\!\!\!\!\!\!\!
 = \sum_{r,l=1}^k \langle E_{v,j}^* W_l ^* W_r E_{v,j} \eta, \eta \rangle + 
\sum_{r,l=1}^k \langle
(I-E_{v,j})^* W_l^* W_r (I- E_{v,j}) \eta, \eta \rangle\\
&\!\!\!\!\!\!\!\!\!\!\!\!\!\!\!\!\!\!\!\!
 = 
\sum_{r,l=1}^k \langle E_{v,j} W_l ^* W_r \eta, \eta \rangle + 
\sum_{r,l=1}^k \langle
(I-E_{v,j}) W_l^* W_r \eta, \eta \rangle\\
&\!\!\!\!\!\!\!\!\!\!\!\!\!\!\!\!\!\!\!\!
 = 
\sum_{r,l=1}^k \langle  W_l ^* W_r \eta, \eta \rangle = \left\| \sum_{r=1}^k W_r \eta\right\|^2.
\end{align*}
Thus, the operator $F_{v,j}$ on $\cl H$ given by 
\[ 
  F_{v,j}\left(\sum_{r=1}^k W_r \eta\right) = \sum_{r=1}^k W_r E_{v,j} \eta
\]
is a well-defined contraction. 

Using that the $E_{v,j}$'s form a PVM, it follows that $F_{v,j}^2=F_{v,j}= F_{v,j}^*$ and $\sum_{j=1}^c F_{v,j} = I$, 
{\it i.e.}, the $F_{v,j}$'s also form a PVM.
Clearly, $F_{v,j} \eta = E_{v,j} \eta$.  Also, 
$$F_{v,j}E_{w,i} (W \eta) = E_{w,i} (WE_{v,j} \eta) = E_{w,i}(F_{v,j} W \eta)$$ 
whenever $W$ is a word on $\{E_{v,i} : v\in V, i = 1,\dots,c\}$,
which shows that $F_{v,j}E_{w,i} = E_{w,i} F_{v,j}$. 

The fact that $E_{v,i}F_{w,j} = F_{w,j}E_{v,i}$ 
easily implies the relations (\ref{eq_sexp}). 
\end{proof}

\begin{cor}\label{synchronouscor} 
A correlation $(p(i,j|v,w))_{v,i,w,j}$ belongs to $C^s_{\cc}(n,c)$ (resp. $C_{\rm q}^s(n,c)$, resp. $C^s_{\rm loc}(n,c)$) 
if and only if there exists a C*-algebra(resp. finite dimensional C*-algebra $\cl A$, resp. abelian C*-algebra) $\cl A$,
a tracial state $s: \cl A \to \bb C$ 
and a generating family $\{e_{v,i} : v\in V, i = 1,\dots,c\}$ of projections 
satisfying $\sum_{i=1}^c e_{v,i} = 1$, $v\in V$, such that
\[ p(i,j|v,w) = s(e_{v,i} e_{w,j}), \ \ \ v, w \in V, i,j = 1,\dots,c.\]
\end{cor}
\begin{proof} 
The statement concerning $C^s_{\cc}(n,c)$ is immediate from Theorem \ref{trthm}. 
For the second equivalence, notice that if 
a synchronous correlation belongs to $C^s_{\q}(n,c)$ 
then it admits a realisation $\{(E_{v,i})_{i=1}^c, (F_{w,j})_{j=1}^c,\cl H,\eta\}$
for which $\cl H$ is finite dimensional. Thus, 
the C*-algebra generated by $\{E_{v,i} : v\in V, i = 1,\dots,c\}$ is finite dimensional.  
Conversely, if $\cl A$ is a finite dimensional C*-algebra and $s: \cl A \to \bb C$ is any state, then the GNS 
construction yields a finite dimensional Hilbert space. 
Thus, the operators $E_{v,i}$ and $F_{w,j}$ from Theorem \ref{trthm} 
act on a finite dimensional Hilbert space, and the claim now follows from Theorem~\ref{findimrealisation}.

Finally, if a synchronous correlation belongs to $C_{\rm loc}(n,c),$ then it has a realisation such that the C*-algebra generated by $\{ E_{v,i} : v \in V, i=1\ldots c \}$ is abelian. Conversely, if $\cl A$ is abelian and $s$ is any state, then the GNS construction yields an abelian family of projections $\{ E_{v,i} \}$ and one can set $F_{v,i} = E_{v,i}.$ 
\end{proof}

Let $\{(E_{v,i})_{i=1}^c$, $(F_{w,j})_{j=1}^c,\cl H,\eta\}$ be a realisation of a 
synchronous correlation $(p(i,j|v,w))_{v,i,w,j}$. 
Let $\cl H_0$ be the smallest closed subspace of $\cl H$ containing $\eta$ and 
invariant under the operators $F_{w,j}$, $w\in V$, $j = 1,\dots,c$. 
Since $F_{w,j}$ is selfadjoint, it is reduced by $\cl H_0$.
Thus, $F_{w,j}$ has a diagonal matrix form with respect to the decomposition
$\cl H = \cl H_0 \oplus \cl H_0^{\perp}$. 
Moreover, since $F_{w,j}$ is a projection, the operator $F^0_{w,j} = F_{w,j}|_{\cl H_0}$ 
is a projection and 
$\sum_{j=1}^c F^0_{w,j} = I_{\cl H_0}$, {\it i.e.}, 
$(F^0_{w,j})_{j=1}^c$ is a PVM on $\cl H_0$ for each $w\in V$. 

By equation (\ref{word}), $\cl H_0$ 
reduces the operators $E_{v,i}$. Hence, setting $E^0_{v,i} = E_{v,i}|_{\cl H_0}$, we have 
that $(E^0_{v,i})_{i=1}^c$ is a PVM on $\cl H_0$;  moreover,
\[ E^0_{v,i} F^0_{w,j}= F^0_{w,j}E^0_{v,i}, \ \ \ v,w\in V, i,j = 1,\dots,c.\]
Thus, all the properties of (\ref{commprop}) are satisfied for 
the new family of operators, but in addition $\eta$ is cyclic for the C*-algebra generated 
by $\{F_{w,j} : w\in V, j = 1,\dots,c\}$.

\begin{defn}\label{d_crea} 
Given a synchronous correlation $(p(i,j|v,w)) \in
  C^s_{\qc}(n,c)$, we call a realisation  $\{ (E_{v,i})_{i=1}^c,
  (F_{w,j})_{j=1}^c , \cl H, \eta \}$  \emph{minimal} if
  $\eta$ is a cyclic vector for the C*-algebra generated by the family $\{F_{w,j} : w\in V, j = 1,\dots,c\}$. 
  
  Given a graph $G = (V,E)$,
  a collection $\{ (E_{v,i})_{i=1}^c,(F_{w,j})_{j=1}^c,\cl H, \eta \}$ will be called a \emph{$c$-realisation of $G$}
  provided that it is a realisation of a synchronous correlation
  $(p(i,j|v,w))_{v,i,w,j}$ that belongs to the kernel of the functional
  $L_{G,c}$.  We will refer to the collection as a \emph{minimal}
  $c$-realisation of $G$ provided it is also a minimal realisation.
\end{defn}

The discussion preceding Definition \ref{d_crea} 
shows how to obtain a minimal $c$-realisation from any $c$-realisation.

\begin{prop}\label{zeroproducts}  
Let $(p(i,j|v,w))_{v,i,w,j} \in C^s_{\cc}(n,c)$
be a synchronous correlation with a minimal realisation 
$((E_{v,i})_{i=1}^c,(F_{w,j})_{j=1}^c ,\cl H,\eta)$. Then the following are equivalent:
\begin{itemize}
\item[(i)] $\langle E_{v,i}F_{w,j} \eta, \eta \rangle =0$,
\item[(ii)] $E_{v,i}E_{w,j}= 0$,
\item[(iii)] $F_{v,i}F_{w,j} =0$. 
\end{itemize}
If this family is a minimal $c$-realisation of a graph $G$ on $n$
vertices, then for every edge $(v,w)$ of $G$ we have that
\[E_{v,i}E_{w,i} = F_{v,i}F_{w,i} =0, \ \ \ i = 1,\dots,c.\]
\end{prop}
\begin{proof} 
We prove the equivalence of (i) and (ii). 
The equivalence of (i) and (iii) is identical. If (ii) holds then, by (\ref{eq_sameon}), 
$\langle E_{v,i} F_{w,j} \eta, \eta \rangle = \langle E_{v,i} E_{w,j} \eta, \eta \rangle =0$. 

Conversely, if (i) holds, then
\begin{align*}
\|E_{v,i}E_{w,j} \eta\|^2 & = \langle E_{v,i} E_{w,j} \eta, E_{v,i}E_{w,j} \eta \rangle =\langle E_{v,i}F_{w,j} \eta, E_{v,i}F_{w,j} \eta \rangle\\
& = \langle E_{v,i}F_{w,j} \eta, \eta \rangle =0,
\end{align*}
since the operators $E_{v,i}$ and $F_{w,j}$ are commuting projections.

Next, for any vector $\xi$ of the form $\xi = F_{w_1,j_1} \cdots F_{w_k,j_k} \eta$, we have that
\[ E_{v,i}E_{w,j}\xi = F_{w_1,j_1} \cdots F_{w_k,j_k} E_{v,i} E_{w,j} \eta =0.\]
Part (ii) now follows by minimality.

Finally, the statement involving graphs follows from the 
equivalence of (i) and (ii) and the fact that if $(v,w)$ is an
edge, then $\langle E_{v,i}F_{w,i} \eta, \eta \rangle = 0$.
\end{proof}

Note that when $p(i,j|v,w) = \langle E_{v,i}F_{w,j} \eta, \eta \rangle,$ then
\[ \sum_j p(i,j|v,w) = \langle E_{v,i} \eta, \eta \rangle := p_A(i|v) \]
is independent of $w$ and represents the marginal probability that Alice produces outcome
$i$ given input $v.$  Similarly, $\sum_i p(i,j|v,w) = \langle F_{w,j} \eta, \eta \rangle := p_B(j|w)$ represents the marginal probability of Bob producing outcome $j$ given input $w$.

\begin{prop}\label{constantmarginals} 
Let $G = (V,E)$ be a graph on $n$ vertices that admits a $c$-realisation.
Then there exists a minimal $c$-realisation $((E_{v,i})_{i=1}^c,(F_{w,j})_{j=1}^c ,\cl H$, $\eta)$ 
of $G$ such that the marginal probabilities satisfy 
\begin{equation}\label{eq_1overc}
\langle E_{v,i} \eta, \eta \rangle = \langle F_{w,j} \eta, \eta \rangle = \frac{1}{c}
\end{equation}
for every $v,w\in V$ and every $i,j = 1,\dots,c$.

Moreover, if $G$ admits a $c$-realisation for which the 
corresponding synchronous correlation is in $C_{\rm x}(n,c),$ for 
${\rm x} \in \{{\rm loc}, {\rm q}, {\rm qa}\}$,
 then a minimal $c$-realisation
$((E_{v,i})_{i=1}^c,(F_{w,j})_{j=1}^c ,\cl H,\eta)$ with marginal probabilities equal to $\frac{1}{c}$ can be chosen so that the corresponding synchronous correlation is in $C_{\rm x}(n,c).$
\end{prop}
\begin{proof} 
Let $\{(E_{v,i})_{i=1}^c, (F_{w,j})_{j=1}^c, \cl H, \eta \}$ be a $c$-realisation of $G$.
Let $\tilde{\cl H}$ be the direct sum of $c$ copies of $\cl H$ and set
\[\tilde{E}_{v,i} = E_{v,1+i} \oplus \cdots \oplus E_{v,c+i} \text{ and }
\tilde{F}_{w,j} = F_{w,1+j} \oplus \cdots \oplus F_{w, c+j}, \]
where the addition in the set of indices is performed modulo $c$.
Set $\tilde{\eta} = \frac{1}{\sqrt{c}} ( \eta \oplus \cdots \oplus \eta)$. 

It is easy to check that $\{ (\tilde{E}_{v,i})_{i=1}^c, (\tilde{F}_{w,j})_{j=1}^c, \tilde{\cl H}, \tilde{\eta}\}$
is a $c$-realisation of $G$. Moreover, for all $v\in V$ and all $i = 1,\dots,c$, we have 
$$\langle \tilde{E}_{v,i} \tilde{\eta}, \tilde{\eta} \rangle = 
\frac{1}{c}\sum_{k=1}^c \langle E_{v,k}\eta,\eta \rangle = \frac{1}{c};$$
similarly, 
\[ \langle \tilde{F}_{w,j} \tilde{\eta}, \tilde{\eta} \rangle = \frac{1}{c}\]
for all $w\in V$, $j = 1,\dots,c$.
The proof is complete after passing  to a minimal $c$-realisation, as described before Definition \ref{d_crea}.

Suppose that the original synchronous correlation belongs to $C_{\rm q}(n,c)$, 
then it has a $c$-realisation of $G$ whose Hilbert space is finite dimensional. 
Then the procedure described in the previous two paragraphs yields a finite dimensional 
Hilbert space, which shows that the graph $G$ admits a $c$-realisation 
that satisfies (\ref{eq_1overc}) and whose synchronous correlation belongs to $C_{\rm q}(n,c)$. 

If the original correlation belongs to $C_{\rm loc}(n,c)$, then all PVM's realising the given correlations can be chosen to commute with each other and the described procedure yields a commuting family of operators, and hence the claim follows. 

Finally, suppose that the original synchronous correlation of a $c$-realisation of $G$  belongs to 
$C_{\rm qa}(n,c)$. Let $(p_k)_{k\in \bb{N}}$ be a sequence of correlations that belong to $C_{\rm q}(n,c)$
such that $\lim_k p_k(i,j|v,w)= p(i,j|v,w)$ for all $i,j,v,w$. 
By the above construction, the correlations, defined by 
$$\tilde{p}_k(i,j|v,w)= \frac{1}{c}\sum_{l=1}^c p_k(i+l,j+l|v,w),$$ 
belong to $C_{\rm q}(n,c)$ and have constant marginals.  Moreover,  
$\lim_k \tilde{p}_k(i,j|v,w) = \frac{1}{c}\sum_{l=1}^c p(i,j|v,w) := \tilde{p}(i,j|v,w).$  
Thus, $\tilde{p} \in C_{\rm qa}(n,c)$ has constant marginals.  Finally, the fact that $L_{G,c}(p) =0$ implies that $L_{G,c}(\tilde{p}) =0.$ 
\end{proof}

D.\ Roberson and L.\ Man\v{c}inska~\cite{arxiv:1212.1724} define the 
\emph{projective rank}  $\xi_{\f}(G)$ of a graph $G$ to be the infimum of the numbers
$\frac{d}{r}$ such that there exists a Hilbert space of (finite) dimension $d$ and projections $E_v$, $v\in V$, 
all of rank $r$, such that $E_vE_w=0$ whenever $(v,w)$ is an edge of $G$; such a
collection is called a \textit{$d/r$-projective representation} of $G$.
Recall that the functional $\tr(X) = \frac{1}{d} {\rm Tr}(X)$, where ${\rm Tr}$ 
is the usual trace on $M_d$,
is the unique tracial state on $M_d$, and that if $E_v$ is a projection of rank $r$ then  
$\tr(E_v) = \frac{r}{d}$. 
Thus, $\xi_{\f}(G)^{-1}$ is the supremum of quantities of the form $s(E_v)$, 
over a set of tracial states $s$ of matrix algebras.
This viewpoint motivates the following definition.

\begin{defn}\label{def:xi_tr}
Let $G = (V,E)$ be a graph. We define the \emph{tracial rank} $\xi_{\tr}(G)$ of $G$ 
to be the reciprocal of the supremum of the set of real numbers 
$u$ for which there exists a unital C*-algebra $\cl A$, a tracial state $s$ on $\cl A$
and projections $e_v\in \cl A$, $v \in V$, such that $e_v e_w=0$ whenever $(v,w) \in E$ 
and $s(e_v) = u$ for every $v \in V$. 
\end{defn}

\begin{prop}\label{fdtr=proj} 
Let $G$ be a graph. 
Then $\xi_{\f}(G)$ is equal to 
the reciprocal of the supremum of the set of real numbers 
$u$ for which there exists a finite dimensional C*-algebra $\cl A$, a tracial state $s$ on $\cl A$
and projections $e_v\in \cl A$, $v \in V$, such that $e_v e_w=0$ whenever $(v,w) \in E$ 
and $s(e_v) = u$ for every $v \in V$.
\end{prop}
\begin{proof} 
Let $\cl U$ be the set of all positive real numbers $u$
for which there exists a finite dimensional C*-algebra $\cl A$, 
a tracial state $s$ on $\cl A$
and projections $e_v\in \cl A$, $v \in V$, such that $e_v e_w=0$ whenever $(v,w) \in E$ 
and $s(e_v) = u$ for every $v \in V$.
Set $U= \sup \cl U$.  
By the paragraph preceding Definition \ref{def:xi_tr}, 
we see that each $r/d$ appearing in the definition of $\xi_{\f}(G)$ is in $\cl U$, and hence $\xi_{\f}(G)^{-1} \le U$.

Let $u \in \cl U$ and $\cl A$ be a finite dimensional C*-algebra as in the previous paragraph. 
Then $\cl A$ is *-isomorphic to a direct sum of matrix algebras, say, 
$\cl A\cong \sum_{l=1}^L \oplus M_{d_l}$, and every tracial state on $\cl A$ has the form
\[s\left(\oplus_{l=1}^L X_l\right) = \sum_{l=1}^L p_l \tr(X_l),\]
for some $p_l \ge 0$ with $\sum_{l=1}^L p_l = 1$. Set $q_l = \frac{p_l}{d_l}$, $l = 1,\dots,L$. 

Each projection $e_v$ is of the form $e_v= \oplus_{l=1}^L e_v^l$, where $e_v^l$ is a 
projection in $M_{d_l}$, and
\[u = \sum_{l=1}^L p_l \frac{{\rm rank}(e_v^l)}{d_l} = \sum_{l=1}^L q_l \, {\rm rank}(e_v^l).\]
Moreover, $\sum_{l=1}^L q_l d_l =1$.

Let 
\[u^{\prime} = \max \{t : \mbox{ there exist } q_l \ge 0, l = 1,\dots,L, \mbox{ such that } \sum_{l=1}^L q_ld_l =1$$
$$\mbox{ and } \sum_{l=1}^L q_l {\rm rank}(e_v^l) = t \mbox{ for all } v\in V\}.\] 
By the previous paragraph, $u \le u^{\prime}$.

Since the coefficients of the constraint equations are all integers, the  maximum $u^{\prime}$ will be attained 
at an $L$-tuple $(q_1,\dots,q_L)$ whose entries are rational. Writing
$q_l = m_l/d$ for some integers $d$ and $m_l$, $l = 1,\dots,L$, and setting
\[e^{\prime}_v = \oplus_{l=1}^L e_v^l \otimes I_{m_l}\]
we obtain a set of projection matrices of size
\[\sum_{l=1}^L m_l d_l =d\] satisfying the required relations and
such that
\[{\rm rank}(e^{\prime}_v) =  {\rm Tr}(e^{\prime}_v) =d \sum_{l=1}^L \frac{m_l}{d} \, {\rm rank}(e_v^l) = d u^{\prime}.\]

Hence, $u^{\prime} \le \xi_{\f}(G)^{-1}$ and it follows that $U \le \xi_{\f}(G)^{-1}$ so that the proof is complete.
\end{proof}

The following is the analogue of the inequality $\xi_{\f}(G)\le
\chi_{\q}(G)$ established in~\cite{arxiv:1212.1724}.

\begin{thm}\label{th_trqc}
Let $G$ be a graph. Then $\xi_{\tr}(G) \le \chi_{\qc}(G)$.
\end{thm}
\begin{proof} 
Given any $c$-realisation of $G$, 
Proposition \ref{constantmarginals} shows that 
there exists a $c$-realisation $((E_{v,i})_{i=1}^c,(F_{w,j})_{j=1}^c ,\cl H,\eta)$ of $G$ 
such that $\langle E_{v,i}\eta,\eta\rangle = c^{-1}$ for all $v\in V$ and all $i = 1,\dots,c$.
By Proposition \ref{zeroproducts}, 
$E_{v,1}E_{w,1}=0$ when $(v,w)$ is an edge of $G$. Thus, $c^{-1} \le \xi_{\tr}(G)^{-1}$ and the proof is complete.
\end{proof}

\section{Graph homomorphisms and projective ranks}
\label{s_ghpr}
Recall that 
we set $C_{\loc}(n,c) = {\rm Loc}(n,c)$, 
$C_{\q}(n,c) = Q(n,c)$ and $C_{\qa}(n,c)= Q(n,c)^-$. Then
for ${\rm x} \in \{ \textrm{loc, q, qa, qc} \}$ and a graph $G$ on $n$ vertices, we have that 
$\chi_{\xx}(G) \le c$ if and only if there exists $A \in C_{\rm x}(n,c)$ such that $L_{G,c}(A) = 0$.


The condition $L_{G,c}((p(i,j|v,w))_{v,i,w,j})=0$ can more compactly be written as
\[p(i = j|v = w)=1 \text{ and } p(i = j|v \sim w) =0,\]
where 
$p(i = j| v \sim w) = 0$ means that  $p(i = j|v,w) = 0$ whenever $(v,w)\in E(G)$.
If we write $p(i,j|v,w) = \langle E_{v,i}F_{w,j}\eta,\eta\rangle$, where 
$(E_{v,i})_{i=1}^c$ and $(F_{w,j})_{j=1}^c$ are mutually commuting PVM's on a Hilbert space $\cl H$, 
$v,w\in V$, and $\eta\in \cl H$ is a unit vector, then we have that
\[p_A(i|v) := \sum_{j=1}^c p(i,j|v,w)= \langle E_{v,i} \eta, \eta \rangle \]
does not depend on $w$ and $j$; 
a similar statement holds for $p_B(j|w)$.

\begin{remark}\label{constantmarginalsrem} 
In the notation introduced above, 
Proposition~\ref{constantmarginals} shows that, for any
${\rm x}\in \{ \textrm{loc, q, qa, c} \}$, if $G$ is a graph on 
$n$ vertices and the correlation $(p(i,j|v,w))_{v,i,w,j} \in C_{\xx}(n,c)$ satisfies 
$p(i=j|v=w)=1$ and $p(i=j| v \sim w) =0$, then there is a correlation
$(p'(i,j|v,w))_{v,i,w,j} \in C_{\xx}(n,c)$
additionally satisfying $p_A'(i|v)= p_B'(j|w) = c^{-1}$ for every $v,w,i,j$.
\end{remark}

We recall the following characterization of points in $C_{\loc}(n,c):$ 
   \begin{align}\label{eq:L}
       C_{\loc}(n,c) &= {\rm Loc}(n,c) \\ 
              &=\Bigl\{ (p(i,j|v,w))_{v,i,w,j} : p(i,j|v,w) =
                         \sum_k \lambda_k \delta(i=f_k(v)) \delta(j=g_k(w)) \Bigr. , \\
              &\Bigl. \text{ for some } \lambda_k > 0 \mbox{ with } \sum_k \lambda_k = 1 
                      \text{ and some }  f_k, g_k : V \to \{1,\dots,c\} \Bigr\}.
    \end{align}

\noindent (Here, the $\delta$ function evaluates to 1 when its condition
argument is true and 0 otherwise, like the Iverson bracket.)

\begin{definition}\label{def:homgame}
   Let $G = (V(G),E(G))$ and $H = (V(H),E(H))$ 
   be graphs on $n$ and $m$ vertices, respectively. For ${\rm x} \in \{\loc, {\rm q}, \qa, \cc\}$ 
   write $G \homX H$
    if there is a correlation
    $(p(i,j|v,w))_{v,i,w,j} \in C_{\xx}(n,m)$ with $v,w \in V(G)$ and $i,j \in V(H)$ such that
    \begin{align*}
        p(i=j|v=w) &= 1 \\
        p(i \sim_H j | v \sim_G w) &= 1,
    \end{align*}
 where $p(i \sim_H j| v,w) := \sum_{(i,j) \in E(H)} p(i,j|v,w)$ and $p(i \sim_H j |v \sim_G w)=1$ means that $p(i \sim_H j | v,w) =1$ whenever $(v,w) \in E(G).$
     
    We will say that such a $(p(i,j|v,w))_{v,i,w,j}$ is an $\xx$-homomorphism from $G$ to $H$.
\end{definition}

Stated briefly, the above conditions are the requirement that $p$ be synchronous and that whenever the inputs $v$ and $w$ are adjacent in $G$, then, with probability 1, the output pair $(i,j)$ is adjacent in $H$.

We will sometimes write $G \to H$ for $G \homL H$, since it can be shown that this corresponds to the classical
definition of a graph homomorphism.
The homomorphism variant $G \homq H$ has been extensively studied in
\cite{roberson2013variations} and \cite{arxiv:1212.1724}.
The following is immediate from the definitions of \cite{pt_chrom} and \cite{cnmsw}:

We let $K_c$ denotes the complete graph on $c$ vertices, i.e.,  $(i,j) \in E(K_c)$ for all $i \ne j$.

\begin{prop}\label{p_xlq}
Let $G$ be a graph. For ${\rm x} \in \{\loc, {\rm q}, \qa, \cc\}$, 
we have that $\chi_{\xx}(G) = \min \{ c: G \homX K_c\}$.
\end{prop}



Let us denote by $\Gc$ the 
complementary graph of $G$, that is, the graph whose vertex set 
coincides with that of $G$ and for which $(v,w)$ is an edge precisely when 
$(v,w)$ is not an edge of $G$ (here it is assumed that $v\neq w$).
Proposition \ref{p_xlq} motivates us to define, for $\xx \in \{\loc, {\rm q}, \qa, \cc\}$,
\[
  \alpha_{\xx}(\Gc) = \omega_{\xx}(G) = \max\{ c : K_c \homX G \}.
\]
The parameters $\omega_{\xx}(G)$ are \emph{quantum clique numbers} of $G$ and 
are complementary to the corresponding chromatic numbers $\chi_{\xx}(G)$.
They will not be used later on in this paper.
Note, however, that $\omega_{\loc}(G)$ (resp. $\alpha_{\loc}(G)$) coincides with the classical 
clique number $\omega(G)$ (resp. independence number $\alpha(G)$) of $G$.




\begin{definition}\label{def:xiX} 
Let $G$ be a graph on $n$ vertices.
    For $\xx \in \{\loc, {\rm q}, \qa, \cc\}$, let $\xi_{\xx}(G)$ be the infimum of the 
    positive real numbers
    $t$ such that there exists $(p(a,b|v,w))_{v,a,w,b} \in C_{\xx}(n,2)$ 
    satisfying
    \begin{align*}
        p(a=b|v=w) &= 1
        \\ p(a=1,b=1 | v \sim w) &= 0
        \\ p(a=1|v) &= t^{-1}.
    \end{align*}
\end{definition}

Note that it makes sense, in the above definition, 
to use only $v$ in the third condition since $C_{\xx}(n,2)$ is
non-signaling.

Suppose that 
\[
  p' = (p'(i,j|v,w))_{v,i,w,j}\in C_{\cc}(n,c) \ \mbox{ and } \ 
  p'' = (p''(a,b|i,j))_{i,a,j,b}\in C_{\cc}(c,l).
\]
We let $p''p' = ((p''p')(a,b|v,w))_{v,a,w,b}$ be the matrix whose entries are given by 
\[
  (p''p')(a,b|v,w) = \sum_{i,j = 1}^c p''(a,b|i,j) p'(i,j|v,w).
\]
Thus, if $p'$ (resp. $p''$) is considered as an element of $M_{c^2,n^2}$ (resp. $M_{l^2,c^2}$), 
whose rows are indexed by the pairs $(a,b)$ (resp. $(i,j)$) and whose columns -- by the pairs
$(i,j)$ (resp. $(v,w)$), then $p''p'$ is the matrix product of $p''$ and $p'$.

\begin{lemma}\label{l_mul}
Let $\xx \in \{\loc, {\rm q}, \qa, \cc\}$. 
\begin{itemize}
\item[(i)] If $p'\in C_{\xx}(n,c)$ and $p''\in C_{\xx}(c,l)$ then 
$p''p'\in C_{\xx}(n,l)$;

\item[(ii)] If $p'\in C_{\xx}(n,c)$ and $p''\in C_{\loc}(c,l)$ then $p''p'\in C_{\xx}(n,l)$.
\end{itemize}
\end{lemma}
\begin{proof}
(i) Assume first that $\xx = \cc$.
Suppose that $\cl H'$ (resp. $\cl H'')$ is a Hilbert space, $\eta\in \cl H'$ (resp. $\eta\in \cl H''$)
is a unit vector and $(E'_{v,i})_{i=1}^c$ and $(F'_{w,j})_{j=1}^c$ 
(resp. $(E''_{i,a})_{a=1}^l$ and $(F''_{j,b})_{b=1}^l$) are mutually commuting PVM's 
such that 
\[
  p'(i,j|v,w) = \langle E'_{v,i}F'_{w,j}\eta',\eta'\rangle \ \mbox{ (resp. } 
  p''(a,b|i,j) = \langle E''_{i,a}F''_{j,b}\eta'',\eta''\rangle\mbox{)},
\]
for all $v,w,i,j,a,b$. 
Let 
$\cl H = \cl H''\otimes\cl H'$, $\eta = \eta''\otimes\eta'$,
\[
  E_{v,a} = \sum_{i=1}^c E''_{i,a}\otimes E'_{v,i} \ \mbox{ and } \ 
  F_{w,b} = \sum_{j=1}^c F''_{j,b}\otimes F'_{w,j}.
\]
It is clear that $(E_{v,a})_{a=1}^l$ and $(F_{w,b})_{b=1}^l$
are mutually commuting POVM's for all $v$ and $w$. Moreover, 
\begin{align*}
\langle E_{v,a}F_{w,b}\eta,\eta\rangle 
& =
\sum_{i,j=1}^c \langle (E''_{i,a}\otimes E'_{v,i})(F''_{j,b}\otimes F'_{w,j})(\eta''\otimes\eta'),(\eta''\otimes\eta')\rangle\\
& = 
\sum_{i,j=1}^c \langle E''_{i,a} F''_{j,b}\eta'',\eta''\rangle  \langle E'_{v,i}F'_{w,j}\eta',\eta'\rangle\\
& = \sum_{i,j=1}^c p''(a,b|i,j) p'(i,j|v,w) = (p''p')(a,b|v,w).
\end{align*}
It follows that $p''p'\in C_{\cc}(n,l)$. 


The arguments given above also apply in the case $\xx = {\rm q}$.
The claim concerning $\xx = \qa$ follows from the fact that 
$C_{\qa}(n,c) = \overline{C_{{\rm q}}(n,c)}$ for all $n$ and $c$. 
The case $\xx = \loc$ follows from the observation preceding Proposition \ref{p_chqaqc}. 

(ii) follows from (i) and the fact that $C_{\loc}(c,l)\subseteq C_{\xx}(c,l)$. 
\end{proof}


\begin{theorem}
    \label{thm:xiX_monotone}
    For $\xx\in \{\loc, {\rm q}, \qa, \cc\}$, we have that $\xi_{\xx}(G) \le \chi_{\xx}(G)$.
    Moreover, if $G \homX H$ then $\xi_{\xx}(G) \le \xi_{\xx}(H)$.
\end{theorem}
\begin{proof}
    Let $(p(i,j|v,w))_{v,i,w,j}$ be an $\xx$-homomorphism from $G$ to $K_c$ with $c=\chi_{\xx}(G)$.
    By Proposition \ref{constantmarginals}, we may assume that $p(i|v) = \frac{1}{c}$ for all $i$ and all $v$. 
Let $p_A'(a|i)$ (resp. $p_B'(b|j)$) be the probability distribution given by 
$p_A'(1|1) = 1, p_A'(0|1)=0$ (resp. $p_B'(1|1) = 1, p_B'(0|1)=0$) and $p_A'(1|i) = 0, p_A'(0|i)=1$ 
(resp. $p_B'(1|j) = 0, p_B'(0|j)=1$) if $i \neq 1$ (resp. $j\neq 1$). 
Set 
\[
  p'(a,b|i,j) = p_A'(a|i) p_B'(b|j), \ \ \ \ a,b = 0,1, i,j = 1,\dots,c.
\]
It is clear that $p'\in C_{\xx}(c,2)$.
By Lemma \ref{l_mul}, $p'p\in C_{\xx}(n,2)$. It remains to check that $p'p$ 
satisfies the conditions of Definition \ref{def:xiX}.

Suppose that $a\neq b$. If $i =j$ then $p'(a,b|i,j) = 0$, while if $i \neq j$ then 
$p(i,j|v,v) = 0$. It follows that $(p'p)(a,b|v,v) = 0$ for all $v$.
Suppose that $v\sim w$. Then $p(i,i|v,w) = 0$ for all $i$, while $p'(1,1|i,j) = 0$ if $i\neq j$. 
It follows that $(p'p)(1,1|v,w) = 0$.
Finally, for fixed $v$, we have
\begin{align*}
(p'p)(1|v) & = \sum_{i,j=1}^c p'(1,0|i,j)p(i,j|v,v) + p'(1,1|i,j)p(i,j|v,v)\\
& = \sum_{i,j=1}^c p_A'(1|i) p_B'(0|j) p(i,j|v,v) + p_A'(1|i) p_B'(1|j) p(i,j|v,v)\\
& = \sum_{j=1}^c p_B'(0|j) p(1,j|v,v) + p_B'(1|j) p(1,j|v,v)\\
& = \sum_{j=1}^c p(1,j|v,v) = \frac{1}{c}.
\end{align*}

We now show the monotonicity of $\xi_{\xx}$. Suppose that $G \homX H$  and 
let 
$$(p(i,j|v,w))_{v,i,w,j}\in C_{\xx}(|V(G)|,|V(H)|)$$ 
be as in Definition \ref{def:homgame}.
Let also 
$(p'(a,b|i,j))_{i,a,j,b} \in C_{\xx}(|V(H)|,2)$ satisfy the three equations of Definition \ref{def:xiX} for the graph $H$. 
Suppose that $v\in V(G)$ and $a\neq b$. Then, 
if $i\neq j$ we have that $p(i,j|v,v) = 0$, while $p'(a,b|i,i) = 0$. Thus,
$(p'p)(a,b|v,v) = 0$. Suppose that $(v,w)\in E(G)$. If $(i,j)\not\in E(H)$ then 
$p(i,j|v,w) = 0$, while if $(i,j)\in E(H)$ then $p'(1,1|i,j) = 0$; thus, 
$(p'p)(1,1|v,w) = 0$. 
Finally, 
\begin{align*}
(p'p)(1|v) & = \sum_{i,j=1}^c p'(1,0|i,j)p(i,j|v,v) + p'(1,1|i,j)p(i,j|v,v)\\
& = \sum_{i=1}^c p'(1,0|i,i)p(i,i|v,v) + p'(1,1|i,i)p(i,i|v,v)\\
& = \sum_{i=1}^c p'(1,1|i,i)p(i,i|v,v) =  \frac{1}{t} \sum_{i=1}^c p(i,i|v,v) = \frac{1}{t},
\end{align*}
for all $v\in V(G)$. By Lemma \ref{l_mul}, $p'p\in C_{\xx}(|V(G)|,2)$. 
Thus, $\xi_{\xx}(G) \le \xi_{\xx}(H)$. 
\end{proof}

\begin{lemma}
    \label{thm:Qq_synchronous_outputs}
    Suppose that $p(i,j|x,y) \in C_{\q}(n,2)$(respectively, $C_{\cc}(n,2)$) satisfies $p(i=j|x=y)=1$.
    Then there exist a finite dimensional C*-algebra $\cl A$ (resp. a C*-algebra $\cl A$),
    a tracial state $s: \cl A \to \bb C$ and projections $E_{v,i}\in \cl A$, $v\in V$, $i = 1,2$, such that
    \begin{enumerate}
        \item $\sum_i E_{v,i} = I$, $v\in V$;
        \item $p(i,j|v,w) = s(E_{v,i} E_{w,j})$ for all $v,w\in V$ and all $i,j = 1,2$;
        \item $E_{v,i} E_{w,j} = 0$ if and only if $p(i,j|v,w)=0$.
    \end{enumerate}
\end{lemma}
\begin{proof} The existence of the C*-algebra, tracial state  and corresponding
  operators follow from the fact that the state is synchronous,
  Corollary~\ref{synchronouscor} and Proposition~\ref{zeroproducts}.
\end{proof}

A graph $G$ is said to have an {\em $a/b$-coloring} provided that to each vertex we can assign a $b$ element subset of $\{ 1,...,a \}$ such that whenever two vertices are adjacent, their corresponding subsets are disjoint. The {\em fractional chromatic number of G} is then defined by
$\chi_f(G) = \inf \{ a/b | \text{ G has an a/b-coloring } \}.$  Alternatively, there is a family of graphs known as the {\em Kneser graphs K(a,b), } where each vertex corresponds to a $b$ element subset of an $a$ element set with vertices adjacent when the sets are disjoint, and $\chi_f(G) = \inf \{ a/b:  G \to K(a,b) \}.$  For more discussion of these ideas and proofs see \cite{gr}.

\begin{theorem}\label{thm:xiq_is_xif}
We have that
     \begin{enumerate}
\item $\xi_{\loc}(G)$ is equal to the fractional chromatic number $\chi_{\f}(G)$;
 \item   $\xi_{\q}(G)$ is equal to the projective rank, $\xi_{\f}(G)$;
\item $\xi_{\qc}(G)$ is equal to the tracial rank, $\xi_{\tr}(G)$.
\end{enumerate}
\end{theorem}
\begin{proof}
    To prove $\xi_{\loc}(G) = \chi_{\f}(G)$, 
    colour the graph $G$ with subsets $S_v \subseteq \{1,\dots,p\}$ of size $\abs{S_v}=q$
    where $p/q=\chi_{\f}(G)$ (this is possible since $\chi_{\f}(G)$ can be interpreted in terms of
    homomorphisms to Kneser graphs).
    Consider the following protocol in $C_{\loc}(n,2)$: Alice and Bob receive vertices
    $v$ and $w$.  They use shared randomness to choose $k \in \{1,\dots,p\}$.
    Alice outputs $1$ if $k \in S_v$ while Bob outputs $1$ if $k \in S_w$.
    The corresponding correlation 
    satisfies the conditions of Definition \ref{def:xiX} with $t=p/q$, so $\xi_{\loc}(G) \le \chi_{\f}(G)$.

    Conversely, suppose that $p(a,b|x,y) \in C_{\loc}(n,2)$ satisfies the conditions of
    Definition \ref{def:xiX} for some $t$.
    By~\eqref{eq:L}, we have
    $p(a,b|v,w) = \sum_k \lambda_k \delta(a=f_k(v)) \delta(b=g_k(w))$ with
    $\lambda_k > 0$ and $\sum_k \lambda_k = 1$.
    The condition $p(a=b|v=w)=1$ requires $f_k=g_k$ for all $k$.
    The condition $p(a=1,b=1 | v \sim w)=0$ guarantees that $f_k(v)f_k(w)=0$ for all $k$
    when $v \sim w$; consequently $V_k := \{ v \in V(G) : f_k(v)=1 \}$ is an independent
    set.  Assigning weight $t \lambda_k$ to set $V_k$ gives a fractional colouring of weight
    $t$ (see Section 7.1 of~\cite{gr}).
    Indeed, $\sum_k t \lambda_k = t$ and for each $v \in V(G)$ we have
    $\sum_{V_k\ni v} t \lambda_k = t \sum_k f_k(v) \lambda_k = t p(a=1|v) = 1$.
    So $\chi_{\f}(G) \le \xi_{\loc}(G)$.

    To prove $\xi_{\q}(G) = \xi_{\f}(G)$,
    suppose $p(a,b|v,w) \in C_{\q}(n,2)$ satisfies the conditions of Definition 
    \ref{def:xiX}. In particular it is synchronous. Let
     $\{E_{v,i} : i = 1,2, v\in V\} \in M_d$ be the representation guaranteed by
    \ref{thm:Qq_synchronous_outputs}.
The operators $E_{v,1}$ then satisfy the conditions in Proposition~\ref{fdtr=proj} and hence, referring to the proof,  $t^{-1} \in \cl U$.
    Taking the infimum over all possible $t$ gives $\xi_{\f}(G) \le \xi_{\q}(G)$.

    Conversely, suppose that $(E_v)_{v\in V}$ is a $d/r$-projective representation of $G$.
    Let $\eta = d^{-1/2} \sum_i e_i \ot e_i$, where the $\{e_i\}_{i=1}^d$ 
    is the standard orthonormal basis and set
    $E_{v,1}=E_v$, $E_{v,0} = (I-E_v)$,
    $F_{w,1}=\overline{E}_w$, and $F_{w,0} = (I-\overline{E}_w)$.
    The probability distribution $p(a,b|v,w) = \langle E_{v,a} \ot F_{w,b} \eta, \eta \rangle$
    is feasible for \ref{def:xiX} with value
    $t^{-1} = p(a=1|v) = \langle E_{v,1}\eta,\eta\rangle = d^{-1} \Tr(E_v) = r/d$.
    So $\xi_{\q}(G) \le d/r$.  Taking the infimum over possible values of $d/r$ gives
    $\xi_{\q}(G) \le \xi_{\f}(G)$.

Finally, we prove $\xi_{\qc}(G) = \xi_{\tr}(G)$. It follows from
Lemma~\ref{thm:Qq_synchronous_outputs} that if $t$ is feasible for
$\xi_{\qc}(G)$ then there exists a C*-algebra and tracial state
satisfying the conditions of Definition~\ref{def:xi_tr}.
Thus, $\xi_{\tr}(G) \le \xi_{\qc}(G)$.

Conversely, assume that we have a C*-algebra $\cl A$ a tracial state
$s$ and projections $E_v$ and a real number $u=t^{-1}$ satisfying the
hypotheses of Definition~\ref{def:xi_tr}.
We set $E_{v,1} = E_v$ and $E_{v,0} = I - E_v$.

If we set $p(i,j|v,w) = s(E_{v,i}E_{w,j})$ then by Corollary~\ref{synchronouscor}, we have that $(p(i,j|v,w)) \in C^s_{\qc}(n,2) \subseteq C_{\qc}(n,2)$.
Thus $t$
is feasible for $\xi_{\qc}(G)$.
This shows $\xi_{\qc}(G) \le \xi_{\tr}(G)$ and the proof is complete.
\end{proof}

\begin{thm} If there exists a graph $G$ for which $\xi_{\f}(G)$ is irrational,
  then the closure conjecture is false, and consequently, Tsirelson's
  conjecture is false. In fact, if $G$ is a graph on $n$ vertices with
  $\xi_{\f}(G)$ irrational, then $C_{\q}(n,2) \ne C_{\q}(n,2)^-$.
\end{thm}
\begin{proof} Let $n$ be the number of vertices of $G$. 
By Theorem \ref{thm:xiq_is_xif}, $\xi_{\f}(G)^{-1} = \xi_{\q}(G)^{-1}$ 
and this value is characterized as the infimum of the positive real numbers $t$ over  
the elements $(p(a,b|v,w))_{v,a,w,b} \in C_{\q}(n,2)$ such that $p(a=1|v) = t^{-1}$.
If this infimum was attained, then there would exist a representation 
of $(p(a,b|v,w))_{v,a,w,b}$ {\it via} a finite dimensional C*-algebra.
It follows by the proof of Proposition~\ref{fdtr=proj} 
that the infimum is a rational number.

Hence, if the infimum is attained by a point in $C_{\q}(n,2)$ then it must be rational.
Thus, if $\xi_{\q}(G)$ is irrational, then we must have a point in $C_{\q}(n,2)^-$ that is not in $C_{\q}(n,2)$.
\end{proof}

\begin{cor} 
If there exists a graph $G$ with $\xi_{\tr}(G)$ irrational, then Tsirelson's conjecture is false. In fact, if $G$ has $n$ vertices, then $C_{\qc}(n,2) \ne C_q(n,2)$.
\end{cor}
\begin{proof} If Tsirelson's conjecture is true, then the closure conjecture is true and $\xi_{\tr}(G) = \xi_{\qc}(G) = \xi_{\q}(G) = \xi_{\f}(G)$, contradicting the previous result.
\end{proof}


Given that $\xi_{\loc}(G)=\chi_{\f}(G)$, $\xi_{\q}(G)=\xi_{\f}(G)$, and $\xi_{\qc}(G)=\xi_{\tr}(G)$, we will henceforth use
the more established notation $\chi_{\f}(G)$, $\xi_{\f}(G)$, $\xi_{\tr}(G)$ and will drop
the notation $\xi_{\loc}(G)$, $\xi_{\q}(G)$, $\xi_{\qc}(G)$.

We turn now to a deeper investigation of $\xi_{\tr}(G)$. We first show
that the equality for a feasible value can be relaxed to an inequality.

\begin{theorem}
    \label{thm:xi_tr_ge}
    The number $\xi_{\tr}(G)$ is equal to the minimum of the positive real numbers 
    $t$ such that there exist a Hilbert space $\cl H$, a unit vector $\eta\in \cl H$, 
    a (unital) C*-algebra $\cl A\subseteq \cl B(H)$
    and projections
    $E_v\in \cl A$, $v \in V$, satisfying
    \begin{align}
        &\; \textrm{ the map } X\to s(X)=\langle X\eta,\eta \rangle
        \textrm{ is a tracial state on } \mathcal{A};
        \label{eq:xi_tr_ge_1}
        \\ &\; E_vE_w=0 \textrm{ if } v \sim w;
        \label{eq:xi_tr_ge_2}
        \\ &\; \langle E_v \eta,\eta \rangle \ge t^{-1} \textrm{ for all } v \in V(G).
        \label{eq:xi_tr_ge_3}
    \end{align}
\end{theorem}
\begin{proof}
    Any solution feasible for Definition~\ref{def:xi_tr} induces a solution feasible for the above conditions, 
    with the same value.  Note that, using the GNS construction, we can assume, without loss of generality, that 
    $s(X)=\langle X\eta,\eta \rangle$, $X\in \cl A$, for some unit vector $\eta$.

  Conversely, suppose we have a feasible solution to the above conditions.
  Let $c_v = s(E_v)$, where $c_v\ge t^{-1}$.
  Set $c = \min \{ c_v : v \in V(G) \}$.

Let $\hat{\cl H} = \cl H \otimes L^2(0,1)$, $\hat{\eta} = \eta \otimes \chi_{[0,1]} \in \hat{\cl H}$ 
and, for $0 \le r \le 1$, let $P_r:L^2(0,1) \to L^2(0,1)$ denote the projection onto the subspace $L^2(0,r)$.  
Let $\cl D$ be the multiplication algebra of $L^{\infty}(0,1)$ acting on $L^2(0,1)$,
$\hat{E}_v = E_v \otimes P_{c/c_v}$ and $\hat{s}(X) = \langle X \hat{\eta}, \hat{\eta} \rangle$, 
$X\in \cl A\otimes\cl D$.

The state $\hat{s}$ is tracial because it is the tensor product of two tracial states.
It is easily verified that this new family of projections $\hat{E}_v$ and state $\hat{s}$
satisfy the conditions of Definition \ref{def:xi_tr} with
$\hat{s}(E_v)=r^{-1}$ where $r^{-1} = c \ge t^{-1}$, so that $r \le t$.

Thus, we attain the same infimum if we require equality in \eqref{eq:xi_tr_ge_3} for all $v$.
\end{proof}

We shall refer to a vector and set of operators satisfying (\ref{eq:xi_tr_ge_1})--(\ref{eq:xi_tr_ge_3}) for some $t$ 
a {\it feasible set for $\xi_{\tr}(G)$ with value $t$.}
The following is an adaptation of a proof from~\cite{cmrssw}.

\begin{theorem}
    $\xi_{\tr}(G \djp H) = \xi_{\tr}(G[H]) = \xi_{\tr}(G) \xi_{\tr}(H)$, where $G * H$ is the disjunctive product
    (co-normal product, OR product) and $G[H]$ is the lexicographical product.
\end{theorem}
\begin{proof}
    The inequality $\xi_{\tr}(G[H]) \le \xi_{\tr}(G \djp H)$
    follows from the inclusion $G[H] \subseteq G \djp H$.

    To prove the inequality, $\xi_{\tr}(G \djp H) \le \xi_{\tr}(G) \xi_{\tr}(H)$,
    let $\eta$ and $E_v$ form a set feasible for $\xi_{\tr}(G)$ with value $t$
    and let $\eta'$ and $E'_{v'}$ form a set feasible for $\xi_{\tr}(H)$ with value $t'$.
    Then $\eta \ot \eta'$ and $E_v \ot E'_{v'}$
    form a set feasible for $\xi_{\tr}(G \djp H)$ with value $tt'$.

    {\bf $\xi_{\tr}(G) \xi_{\tr}(H) \le \xi_{\tr}(G[H])$:}
    Let $\eta$ and $E_{g,h}$ be a feasible set for $\xi_{\tr}(G[H])$ with value $t$, where $g \in V(G)$ and $h \in V(H)$.
    For $g \in V(G)$ define $\tilde{E}_g$ to be the projection onto the span of the ranges of
    $\{ E_{g,h} \}_{h \in H}$.

If $g \sim g'$ in $G$ then $(g,h) \sim (g', h')$ for every $h,h' \in V(H)$ and hence
$E_{g,h} E_{g',h'} =0$.  From this it follows that $g \sim g'$ implies that $\tilde{E}_g \tilde{E}_{g'}=0$.
Recall that $\tilde{E}_g$ is the strong limit of $\big( \sum_h E_{g,h} \big)^{1/n}$ and hence
$X\to s(X)=\langle X\eta, \eta \rangle$ is a tracial state on the C*-algebra generated by the $\tilde{E}_g$
(indeed, the latter C*-algebra is a subalgebra of the von Neumann algebra generated by 
the set $\{E_{g,h} : g\in G, h \in H\}$). 
Thus, $\eta$ and $\tilde{E}_g$ satisfy all the conditions to be a feasible set for $G$. Set $c_g=s(\tilde{E}_g)$ and let $r^{-1}= \min \{ c_g: g \in V(G) \}=c_f$ for some $f \in V(G)$ so that we have a feasible set for $G$ with value $r$.

Let $\tilde{\eta}=\sqrt{r} \tilde{E}_f \eta$
and $\tilde{s}(X)=\langle \tilde{\eta}, X \tilde{\eta} \rangle=r s(\tilde{E}_f X \tilde{E}_f)$.
Then $\tilde{s}$ is a state, tracial on the algebra generated by $\{E_{f,h} : h \in V(H)\}$.
Since $\tilde{s}(E_{f,h}) = r s(E_{f,h}) \ge rt^{-1} = (r^{-1} t)^{-1}$,
we see that $\tilde{\eta}$ and $\{\tilde{E}_{f,h}\}_{h \in V(H)}$
is a feasible set for $\xi_{\tr}(H)$ with value $r^{-1} t$.

Thus,
\[ 
  \xi_{\tr}(G) \xi_{\tr}(H) \le r \cdot (r^{-1}t) = t,
\]
and since $t$ was an arbitrary feasible value for $G[H]$, we have
\[
  \xi_{\tr}(G) \xi_{\tr}(H) \le \xi_{\tr}(G[H]).
\]
\end{proof}

\section{A SDP lower bound for the commuting quantum chromatic number}

We now explore a quantity which can be seen either as a semidefinite relaxation of the
quantity $\xi_{\tr}(G)$ or as a strengthening of the quantity $\thpbar(G)$.
As before, we assume that $|V(G)| = n$. For a matrix $Y\in M_{n+1}(\bb{R})$, 
we index the first row and column by $0$, and the rest -- by the elements of the set $V(G)$.  Since we are discussing a semidefinite programming problem in this section, we shall use the more familiar $Y \succeq 0$ to indicate that the matrix $Y$ is positive semidefinite.

\begin{definition}
    Define
    \begin{align}
        \notag \xisdp(G) = \min\Big\{ Y_{00} &\; |
           \, \exists \, Y \in M_{n+1}(\mathbb{R}), Y \succeq 0, Y_{vw} \ge 0,
        \\ &\; \notag Y_{0v} = Y_{vv} = 1 \textrm{ for } v \in V(G),
        \\ &\; \notag Y_{vw} = 0 \textrm{ for } v \sim w,
        \\ &\; \notag \sum_{v \in S} Y_{vw} \le 1
            \textrm{ for $S$ a clique of $G$ and } w \in V(G),
        \\ &\; Y_{00} + \sum_{v \in S} \sum_{w \in T} Y_{vw} \ge \abs{S} + \abs{T}
            \textrm{ for $S, T$ cliques of $G$} \Big\}.
        \label{eq:xi_sdp}
    \end{align}
\end{definition}

\begin{remark} If in the above definition we assumed instead that $Y$ is a complex matrix that satisfied the remaining 
  conditions, then since $Y=Y^*$ we see that $\frac{Y + Y^t}{2}$ is a
  real matrix that also satisfies these equations and has the same
  value for the $(0,0)$-entry. Thus, we obtain the same value for $\xisdp(G)$ if we
  require $Y \in M_{n+1}(\bb R)$ or $Y \in M_{n+1}(\bb C)$.
\end{remark}

\begin{theorem}
We have that
    $\thpbar(G) \le \xisdp(G) \le \xi_{\tr}(G)$.
\end{theorem}
\begin{proof}
    We first show that $\thpbar(G) \le \xisdp(G)$.    
    Let $\evec \in \mathbb{R}^{\abs{V(G)}}$ be the vector all of whose entries are equal to $1$, 
    and $J$ be the matrix all of whose entries are equal to $1$.
    Let $Y$ satisfy conditions~\eqref{eq:xi_sdp}.
    Then $Y$ has the block form
    \begin{align*}
        Y =
        \begin{pmatrix}
            Y_{00} & \evec^T \\ \evec & M
        \end{pmatrix}.
    \end{align*}
    By Cholesky's Lemma, $Y\succeq 0$ if and only if
    $M \succeq Y_{00}^{-1} J$. 
    Define $Z = Y_{00} M - J$; then $Z \succeq 0$, $Z_{vw} \ge -1$ for all $v,w\in V(G)$, 
    $Z_{vw} = -1$ whenever $v \sim w$,
    and $Z_{vv} = Y_{00}-1$ for all $v\in V(G)$.
    So $Z$ is feasible for $\thpbar(G)$ with value $Y_{00}$ (see \cite{cmrssw} for the definition of 
    $\thpbar(G)$). 

    We next show that $\xisdp(G) \le \xi_{\tr}(G)$.
    Let $\eta$ and $E_v$, $v\in V(G)$ be as in Theorem \ref{thm:xi_tr_ge},
    with $t=\xi_{\tr}(G)$.
    Note that by the proof of Theorem~\ref{thm:xi_tr_ge} we can assume equality
    in~\eqref{eq:xi_tr_ge_3}.
    Set $\eta_v = E_v \eta$, $v\in V(G)$, 
    and $Y \in M_{n+1}(\mathbb{R})$
    as $Y_{vw} = t \langle \eta_v, \eta_w \rangle$, where we take $\eta_0 =\eta$.
    We have that $Y_{vw} = ts(E_v E_w)$, $Y_{0v} = ts(E_v) = 1$, and $Y_{00}=ts(I)=t$.

    Conditions \eqref{eq:xi_tr_ge_1}--\eqref{eq:xi_tr_ge_3}, along with the fact that $E_v^2=E_v$, give
    \begin{align*}
        &\; Y_{vv} = ts(E_v E_v) = Y_{0v} = 1
        \\ &\; Y_{vw} = ts(E_v E_w) = 0 \textrm{ for } v \sim w.
    \end{align*}

    Consider a clique $S$ of $G$.
    We have that $E_v \perp E_w$ when $v \sim w$, so $I - \sum_{v \in S} E_v$ is a projection.
    For $S$ a clique of $G$ and $w \in V(G)$,
    \begin{align*}
        1 - \sum_{v \in S} Y_{vw} &= ts(E_w) - \sum_{v \in S} ts(E_v E_w)
        \\ &= ts((I - \sum_{v \in S} E_v) E_w)
        \\ &= ts(E_w (I - \sum_{v \in S} E_v) E_w) \ge 0,
    \end{align*}
    where in the last step we used that $s$ is a tracial state.
    Similarly, for $S,T$ cliques of $G$,
    \begin{align*}
        Y_{00} + \sum_{v \in S} &\sum_{w \in T} Y_{vw} - \abs{S} - \abs{T}
        \\ &= ts(( I - \sum_{v \in S} E_v )( I - \sum_{w \in S} E_w ))
        \\ &= ts(( I - \sum_{v \in S} E_v )( I - \sum_{w \in S} E_w )( I - \sum_{v \in S} E_v ))
        \ge 0.
    \end{align*}

    So $Y$ is a feasible solution to~\eqref{eq:xi_sdp} with value $t$.
\end{proof}

We can now compute the tracial rank of an odd cycle. 

\begin{theorem}
    \label{thm:xi_tr_cycle}
    $\xi_{\tr}(C_{2k+1}) = \frac{2k+1}{k}$.
\end{theorem}
\begin{proof}
    The following is inspired by a proof from~\cite{roberson2013variations}.
    Let $Y$ be feasible for~\eqref{eq:xi_sdp}.
    Considering cliques $S=\{1,2\}$ and $T=\{2,3\}$, we have
    \begin{align}
        Y_{00} + \sum_{v \in \{1,2\}} \sum_{w \in \{2,3\}} Y_{vw} \ge 4
        &\implies Y_{13} \ge 3 - Y_{00}.
        \label{eq:bound_Y13}
    \end{align}
    Let $n = 2k+1$.
    We equip the set $\{1,\dots,n\}$ with addition modulo $n$.
    For any $a,b \in \{1,\dots,n\}$, we have
    \begin{align*}
        - \sum_{v \in \{a,a+1\}} Y_{vb} \ge -1
        &\implies -Y_{a,b} -Y_{(a+1),b} \ge -1
        \\ Y_{00} + \sum_{v \in \{(a+1),(a+2)\}} \sum_{w \in \{b\}} Y_{vw} \ge 3
        &\implies Y_{(a+1),b} + Y_{(a+2),b} \ge 3-Y_{00}.
    \end{align*}
    Adding these inequalities gives
    \begin{align*}
        Y_{(a+2),b} - Y_{a,b} \ge 2-Y_{00}.
    \end{align*}
    Adding together $c$ instances of this inequality gives $Y_{(a+2c),b} - Y_{a,b} \ge c(2-Y_{00})$.
    Taking $a=3,b=1,c=k-1$ gives $Y_{n1} - Y_{31} \ge (k-1)(2-Y_{00})$.
    But $Y_{n1}=0$ since $1 \sim n$.  Adding~\eqref{eq:bound_Y13} gives
    $0 \ge (k-1)(2-Y_{00}) + (3 - Y_{00}) = (2k+1) - k Y_{00}$.
    So $Y_{00} \ge (2k+1)/k$.

Hence,
\[(2k+1)/k \le \xisdp(C_k) \le \xi_{\tr}(C_k) \le \xi_{\loc}(C_k) = \chi_{\f}(C_k).\]
And $\chi_{\f}(C_k) = (2k+1)/k$~\cite{roberson2013variations} so the result follows.
\end{proof}

Note that the preceding proof did not make use of $Y \succeq 0$.
So in fact it would even follow from an LP relaxation of the SDP.

\begin{corollary}
    There is a graph $G$ for which $\lceil \thpbar(G) \rceil < \chi_{\qc}(G)$. 
\end{corollary}
\begin{proof} By \cite[Theorem~17]{cmrssw},
\[\lceil \thpbar(C_5 \djp K_3) \rceil = \lceil \thpbar(C_5) \thpbar(K_3) \rceil =\lceil 3 \sqrt{5} \rceil = 7,\]
while
\[7<\frac{15}{2} = \xi_{\tr}(K_3) \xi_{\tr}(C_5) = \xi_{\tr}(C_5 \djp K_3) \le \chi_{\qc}(C_5 \djp K_3) \le \chi(C_5 \djp K_3) =8.\]
\end{proof}

For all graphs $G$ we have
\[ \xisdp(G) \le \xi_{\tr}(G) \le \xi_{\f}(G) \le \chi_{\f}(G). \]
In the proof of Theorem~\ref{thm:xi_tr_cycle} we see that for odd cycles
$\xisdp(G)=\chi_{\f}(G)$, so this chain of inequalities collapses.
Numerical results show that, in fact, $\xisdp(G)=\chi_{\f}(G)$ for all graphs on 9 vertices or
less.  So now we know the value of $\xi_{\f}$ on all of these graphs (it is equal to
$\chi_{\f}$) whereas before the only nontrivial graphs for which this quantity was known were
the Kneser graphs and odd cycles.



\section*{Acknowledgement}
The authors would like to thank the referee for a diligent reading that improved the exposition considerably.

\end{document}